\documentclass[11pt]{article}

\usepackage[utf8]{inputenc}
\usepackage[T1]{fontenc}
\usepackage{mathpazo}

\RequirePackage[english]{babel}

\RequirePackage[a4paper,margin=2.5cm]{geometry}
\RequirePackage{parskip}

\newcommand{\affiliation}{\footnote}
\newcommand{\affiliationmark}[1][\value{footnote}-1]{\footnotemark[\numexpr#1+1\relax]}
\usepackage{xcolor}
\definecolor{cblue}{RGB}{0,70,140}
\definecolor{cgreen}{RGB}{100,140,0}
\definecolor{cred}{RGB}{190,10,50}

\usepackage[shortlabels]{enumitem}
\setlist{itemsep=0ex,topsep=0ex,parsep=0.4ex}

\usepackage{amsmath,amsfonts,amssymb,amsthm,mathtools,thmtools,thm-restate,bbm,centernot}

\usepackage{tikz}
\usetikzlibrary{arrows.meta,decorations.pathmorphing}
\tikzset{
    graph/.style={
        line width=1.5pt,
        every node/.style={circle,fill,inner sep=0pt,outer sep=1pt,minimum size=5pt}
    },
    snake/.style={
        decorate,
        decoration={snake,segment length=0.3cm,amplitude=0.05cm}
    }
}

\usepackage[hyphens]{url} 
\usepackage[hidelinks,colorlinks,pagebackref]{hyperref} 
\hypersetup{citecolor=cgreen,linkcolor=cblue,urlcolor=cblue}
\usepackage[capitalise,nameinlink,noabbrev,compress]{cleveref} 

\renewcommand*{\backref}[1]{}
\renewcommand*{\backrefalt}[4]{
	\ifcase #1 Not cited.%
	\or $\uparrow$#2%
	\else $\uparrow$#2%
	\fi%
}

\let\oldbibliography\bibliography
\renewcommand{\bibliography}[1]{
  {

    \hypersetup{linkcolor=cred}
    \bibliographystyle{bibstyle}
    \oldbibliography{#1}
  }
}

\theoremstyle{plain}
\newtheorem{theorem}{Theorem}[section]
\newtheorem{lemma}[theorem]{Lemma}

\newtheorem{corollary}[theorem]{Corollary}
\newtheorem{conjecture}[theorem]{Conjecture}
\newtheorem{problem}[theorem]{Problem}

\theoremstyle{definition}
\newtheorem{definition}[theorem]{Definition}

\makeatletter
\renewenvironment{proof}[1][\proofname]
{\par\pushQED{\qed}
	\normalfont\topsep6\p@\@plus6\p@\relax\trivlist
	\item[\hskip\labelsep\bfseries#1\@addpunct{.}]
	\ignorespaces}
{\popQED\endtrivlist\@endpefalse}
\makeatother

\let\emptyset\varnothing
\newcommand{\eps}{\varepsilon}

\newcommand{\pr}{\mathbb{P}}
\newcommand{\ev}{\mathbb{E}}
\newcommand{\cO}{\mathcal O}
\newcommand{\cF}{\mathcal F}
\DeclarePairedDelimiter{\abs}{\lvert}{\rvert}
\DeclarePairedDelimiter{\ceil}{\lceil}{\rceil}
\DeclareMathOperator{\la}{la}
\DeclareMathOperator{\fla}{fla}
\DeclareMathOperator{\End}{End}

\title{New bounds for linear arboricity and related problems}
\author{Micha Christoph\affiliation{Department of Mathematics, ETH Z\"{u}rich, Switzerland (\textsf{\href{mailto:micha.christoph@math.ethz.ch}{micha.christoph@math.ethz.ch}}). Research supported by SNSF Ambizione Grant No. 216071.} \and Nemanja Dragani\'{c}\affiliation{Mathematical Institute, University of Oxford, United Kingdom (\textsf{\{\href{mailto:nemanja.draganic@maths.ox.ac.uk}{nemanja.draganic},\href{mailto:antonio.girao@maths.ox.ac.uk}{antonio.girao},\href{mailto:eoin.hurley@maths.ox.ac.uk}{eoin.hurley},\allowbreak\href{mailto:lukas.michel@maths.ox.ac.uk}{lukas.michel},\href{mailto:alp.muyesser@maths.ox.ac.uk}{alp.muyesser}\}@maths.ox.ac.uk}). Research of Nemanja Dragani\'c supported by SNSF project 217926. Research of Eoin Hurley supported by ERC Advanced Grant 883810.} \and Ant\'{o}nio Gir\~{a}o\affiliationmark[2] \and Eoin Hurley\affiliationmark[2] \and Lukas Michel\affiliationmark[2] \and Alp M\"{u}yesser\affiliationmark[2]}

\date{28 July 2025}

\begin{document}

\maketitle

\begin{abstract}
    A linear forest is a collection of vertex-disjoint paths. The Linear Arboricity Conjecture states that every graph of maximum degree $\Delta$ can be decomposed into at most $\lceil(\Delta+1)/2\rceil$ linear forests. We prove that $\Delta/2 + \cO(\log n)$ linear forests suffice, where $n$ is the number of vertices of the graph. If $\Delta = \Omega(n^\eps)$, this is an exponential improvement over the previous best error term. We achieve this by generalising P\'osa rotations from rotations of one endpoint of a path to simultaneous rotations of multiple endpoints of a linear forest. This method has further applications, including the resolution of a conjecture of Feige and Fuchs on spanning linear forests with few paths and the existence of optimally short tours in connected regular graphs.
\end{abstract}

\section{Introduction}

In 1850, Thomas Kirkman posed the following problem in \textit{The Lady's and Gentleman's Diary}.
\begin{quote}
    Fifteen young ladies in a school walk out three abreast for seven days in succession: it is required to arrange them daily so that no two shall walk twice abreast.
\end{quote}
Put differently, Kirkman was asking for a decomposition\footnote{A \emph{decomposition} of a graph is a partition of its edge set.} of the edges of the complete graph on fifteen vertices into copies of a triangle. Investigating a similar problem in 1882, Walecki proved that for all odd $n$, the complete graph $K_n$ can be decomposed into Hamilton cycles, which are cycles that contain every vertex. What other graphs can $K_n$ be decomposed into? The answer to this question is perhaps counter-intuitive: if basic divisibility conditions are met, then $K_n$ can be decomposed into a wide variety of graphs. Examples include large trees in Ringel's conjecture \cite{montgomery2021proof, keevash2025ringel}, complete graphs in the study of designs \cite{keevash2014existence, glock2023existence}, cycles of various lengths in the Oberwolfach problem \cite{keevash2022generalised}, and many others. Notably, all of these central conjectures have been proved in the past few decades as the machinery associated with absorption and the probabilistic method has become more sophisticated and versatile \cite{alon2016probabilistic, kang2021graph, glock2023existence, montgomery2024transversals, kwan2024high}.

The complete graph, however, is a very special and highly symmetric object, and so it is unclear exactly which of its properties are necessary for these decomposition results. One way to investigate this is to consider dense graphs as a slight weakening of complete graphs. It turns out that here too, a surprising variety of decompositions are possible, and here too there has been significant progress in the past few decades. Perhaps the largest body of work in this direction concerns the decomposition of graphs that satisfy minimum degree conditions. Prominent examples of this include progress on the Nash-Williams conjecture on triangle decompositions \cite{barber2016edge, delcourt2021progress}, generalising Kirkman's problem, and the Hamilton decomposition theorem for regular Dirac graphs \cite{csaba2016proof}, generalising Walecki's result. Another exciting recent direction concerns the decomposition of graphs that are dense and quasirandom. Here, results include generalisations of Ringel's conjecture \cite{keevash2025ringel} and of the Oberwolfach problem \cite{keevash2022generalised}, as well as \emph{approximate} decompositions into arbitrary graphs of bounded degree \cite{kim2019blow}. This work shares some machinery with the case of complete graphs and makes frequent use of Szemer\'edi's famous regularity lemma. Overall, for the decomposition of dense graphs into graphs of fixed size or very sparse spanning structures, the following perspective has emerged.
\begin{quote}
    If there is no obvious barrier to the existence of a particular graph decomposition,\\then the decomposition likely exists.
\end{quote}

But what if we go beyond complete and dense graphs? After all, most graphs encountered in areas like statistical physics, computer science, or group theory are graphs of bounded degree. Is it still true that beyond obvious barriers, decompositions exist? Of course, in general we can no longer hope for decompositions into complete graphs or Hamilton cycles since sparse graphs may not even contain these structures. So, we have to ask for different decompositions.

Two classical and elegant results on the decomposition of sparse graphs are Vizing's theorem \cite{vizing1964estimate} and the Nash-Williams theorem \cite{nash1961edge, nash1964decomposition}. Vizing's theorem, inspired by work of Shannon \cite{shannon1949theorem}, states that any graph of maximum degree $\Delta$ can be decomposed into $\Delta + 1$ matchings. The Nash-Williams theorem states that a graph can be decomposed into $k$ forests if and only if every set of $\ell$ vertices spans at most $k \cdot (\ell-1)$ edges. Thus, in both cases, decompositions exist beyond the obvious barriers. In recent decades, however, progress on sparse decomposition problems has lagged far behind its dense counterparts. This is largely because various tools---such as Szemer\'{e}di’s regularity lemma---offer little leverage in sparse graphs\footnote{Although there exist sparse variants of the regularity lemma \cite{gerke2005sparse,kohayakawa1997szemeredi, scott2011szemeredi}, they only provide information for sparse graphs whose edges are sufficiently well distributed, that is, pseudorandom graphs.}. Among the most famous open problems in this area are the List Edge Colouring Conjecture \cite{vizing1976coloring, erdos1979choosability, kahn2000asymptotics}, the Erd\H{o}s-–Gallai Conjecture on decomposing graphs into cycles and edges \cite{erdos1966representation,bucic2024towards}, and the Linear Arboricity Conjecture.

\subsection{The Linear Arboricity Conjecture}

A \emph{linear forest} is a collection of vertex-disjoint paths, and the \emph{linear arboricity} $\la(G)$ of a graph $G$  is the minimum number of linear forests needed to decompose the edges of $G$ \cite{harary1970covering}. Matchings are a special case of linear forests, and linear forests are a special case of forests, so the linear arboricity of a graph lies between the aforementioned results of Vizing and Nash-Williams. Since Hamilton paths are linear forests with the maximum possible number of edges, linear arboricity can also be viewed as a relaxation of a decomposition into Hamilton paths.

How large is the linear arboricity of a graph? Since a linear forest covers at most two edges incident to any vertex, every graph $G$ with maximum degree $\Delta$ satisfies $\la(G) \ge \ceil{\Delta/2}$. Moreover, a $\Delta$-regular graph $G$ on $n$ vertices has $n \cdot \Delta / 2$ edges while any linear forest of $G$ has strictly less than $n$ edges, and so $\la(G) \ge \ceil{(\Delta+1)/2}$. The celebrated Linear Arboricity Conjecture of Akiyama, Exoo, and Harary \cite{akiyama1980covering} asserts that, beyond these obvious barriers, there should always be a decomposition into linear forests.

\begin{conjecture}[Linear Arboricity Conjecture]
    The linear arboricity of a graph with maximum degree $\Delta$ is at most $\ceil{(\Delta + 1)/2}$.
\end{conjecture}

This conjecture immediately received significant attention \cite{akiyama1981covering,enomoto1981linear,peroche1982partition,tomasta1982note,enomoto1984linear,guldan1986linear,guldan1986some}. By 1988, it was known to hold in the special cases $\Delta \in \{3,4,5,6,8,10\}$, and the best general upper bound was $\la(G) \le \ceil{(3\Delta +2)/4}$. In an influential paper, Alon \cite{alon1988linear} then confirmed the Linear Arboricity Conjecture asymptotically, showing that
\[
    \la(G) \le \frac{\Delta}{2} + \cO\left(\frac{\Delta \log \log \Delta}{\log \Delta}\right)=(1+o(1)) \cdot \frac{\Delta}{2}.
\]
This was one of the early applications of the Lov\'{a}sz Local Lemma, giving the Linear Arboricity Conjecture a prominent place in the history of the probabilistic method \cite{alon2016probabilistic}.

Since then, some additional special cases of the conjecture have been solved, such as graphs with minimum degree at least $(1+o(1)) \cdot n/2$ \cite{glock2016optimal,gao2024linear}, planar graphs \cite{wu1999linear, wu2008linear}, random or dense quasirandom graphs \cite{draganic2025optimal, glock2016optimal}, and random regular graphs of constant degree \cite{mcdiarmid1990linear}. The error term of Alon's upper bound was subsequently improved to $\cO(\Delta^{2/3} \log^{1/3} \Delta)$ by Alon and Spencer \cite{alon2016probabilistic}, to $\cO(\Delta^{2/3-\eps})$ by Ferber, Fox, and Jain \cite{ferber2020towards}, and to $\cO(\sqrt{\Delta} \log^4 \Delta)$ by Lang and Postle \cite{lang2023improved}. A central ingredient in all of these upper bounds is a certain degree of randomisation, either through a random partition of the vertex set or through nibble-based strategies. For both approaches, $\Omega(\sqrt{\Delta})$ is a natural bottleneck for the error term, caused by concentration inequalities used throughout these proofs. This is a reflection of the fact that the standard deviation of a binomial random variable with $\Delta$ trials is $\Theta(\sqrt{\Delta})$. Our main contribution is the following upper bound on the linear arboricity.

\begin{theorem}\label{thm:lineararboricitylogerror}
    Every $n$-vertex graph $G$ with maximum degree $\Delta$ satisfies $\la(G) \le \Delta / 2 + \cO(\log n)$.
\end{theorem}

This result breaks the square-root barrier for the Linear Arboricity Conjecture if $\Delta = \Omega(\log^2 n)$ and represents an exponential improvement in the error term for $\Delta = \Omega(n^\eps)$. Before our work, there were no results that improved on the upper bound of Lang and Postle \cite{lang2023improved} even for very dense graphs with $\Delta = \Omega(n)$.

For the fractional relaxation of linear arboricity, the previous best error term was also stuck at the same square-root barrier. Formally, the \emph{fractional linear arboricity} $\fla(G)$ of a graph $G$ is the smallest real number $k$ for which we can choose a random linear forest of $G$ that contains every fixed edge of $G$ with probability at least $1/k$. Clearly, $\Delta/2 \le \fla(G) \le \la(G)$. Feige, Ravi, and Singh \cite{feige2014short} showed that $\fla(G) \le \Delta/2 + \cO(\sqrt{\Delta})$. We improve the error term in this result to $\cO(\log \Delta)$, providing an exponential improvement for all ranges of $\Delta$.

\begin{theorem}\label{thm:fractionallineararboricitylogerror}
    Every graph $G$ with maximum degree $\Delta$ satisfies $\fla(G) \le \Delta / 2 + \cO(\log \Delta)$.
\end{theorem}

Our proofs use a novel variant of P\'{o}sa rotations for linear forests. We discuss our proof techniques in more detail in \cref{ssec:prooftechniques}. In particular, this section outlines the key lemma that we use to prove our results, \cref{prop:componentsize}. In addition to being crucial for our upper bounds on the linear arboricity, this lemma also has further applications, which we discuss in the following.

\subsection{Spanning linear forests with few paths and short tours}\label{ssec:linearforeststours}

If the Linear Arboricity Conjecture holds, then, by averaging, it follows that every $d$-regular graph\footnote{Recall that any graph with maximum degree $\Delta$ is contained in a $\Delta$-regular graph, and so we may, without loss of generality, assume that a given graph is regular whilst obtaining bounds for the Linear Arboricity Conjecture.} on $n$ vertices has a spanning linear forest with as few as $2 \cdot n/(d+2)$ paths\footnote{See \cite[Section 1.1]{feige2022path} for a detailed discussion of these calculations.}. This consequence of the Linear Arboricity Conjecture is a challenging and well-studied problem in its own right. Magnant and Martin \cite{magnant2009note} conjectured that every $d$-regular graph has a spanning linear forest with $n/(d+1)$ paths. Feige and Fuchs later conjectured that the weaker bound $\cO(n/d)$ should hold. These conjectures have been reiterated several times in recent work \cite{montgomery2024approximate, letzter2025nearly}.

The conjecture of Magnant and Martin would be tight, as witnessed by a disjoint union of cliques. However, this conjecture is only known to hold for $d \le 6$ \cite{magnant2009note,feige2022path}, $d = \Omega(n)$ \cite{han2018vertex,gruslys2021cycle}, and for partitioning almost all vertices of a graph into paths \cite{montgomery2024approximate,letzter2025nearly}. The upper bound on the fractional linear arboricity of Feige, Ravi, and Singh \cite{feige2014short} implies that there is a spanning linear forest with $\cO(n/\sqrt{d})$ paths. Very recently, the authors of the current paper \cite{christoph2025cyclefactors} broke the square-root barrier for this problem, improving the upper bound to $\cO((n\log d)/d)$.

These results employed a variety of different approaches. In addition to strategies also used for the Linear Arboricity Conjecture, such as random partitioning arguments \cite{montgomery2024approximate}, previous techniques included Szemer\'edi's regularity lemma \cite{han2018vertex}, robust expander decompositions \cite{gruslys2021cycle,letzter2025nearly}, and entropy-based methods \cite{christoph2025cyclefactors}. Using our new methods based on rotations of linear forests, we establish the conjecture of Magnant and Martin up to a factor of two. In particular, this proves the conjecture of Feige and Fuchs.

\begin{theorem}\label{thm:smalllinearforest}
    Every $d$-regular $n$-vertex graph has a spanning linear forest with at most $2 \cdot n / (d+1)$ paths.
\end{theorem}

For regular bipartite graphs, we even confirm the conjecture of Magnant and Martin, see \cref{thm:smalllinearforestbipartite}. We prove these results in \cref{sec:smalllinearforests} as an immediate consequence of the key lemma that we use to make progress on the Linear Arboricity Conjecture.

Finally, our results imply that connected $d$-regular graphs have short tours, where a \emph{tour} is a sequence of adjacent vertices that ends at the starting vertex and visits every vertex at least once. The problem of finding such tours was introduced by Vishnoi \cite{vishnoi2012permanent} as a restriction of the Travelling Salesman Problem to regular graphs. Feige, Ravi, and Singh \cite{feige2014short} showed that every connected $d$-regular graph on $n$ vertices has a tour of length $(1+\cO(1/\sqrt{d})) \cdot n$ and observed that one cannot do better than $(1+\cO(1/d)) \cdot n$. The authors of the current paper \cite{christoph2025cyclefactors} very recently showed that tours of length $(1 + \cO((\log d)/d)) \cdot n$ exist. Here, we resolve this problem. This is a direct consequence of \cref{thm:smalllinearforest} in combination with a result of Feige, Ravi, and Singh \cite[Corollary 1]{feige2014short}.

\begin{corollary}
    Every connected $d$-regular $n$-vertex graph has a tour of length at most $(1+\cO(1/d)) \cdot n$. 
\end{corollary}

It is straightforward to construct regular connected graphs that are not Hamiltonian. However, our results show that, at least in some sense, regular graphs are close to being Hamiltonian. 

\subsection{Proof techniques}\label{ssec:prooftechniques}

A rotation of a path or a linear forest is a simple operation that modifies the path or linear forest in such a way that exactly one of its endpoints moves to a different vertex of the graph (see also \cref{fig:linearforestrotations}). This idea can be traced back to Dirac's theorem \cite{dirac1952some} stating that graphs on $n$ vertices with minimum degree at least $n/2$ are Hamiltonian. It was also used by P\'osa \cite{posa1976hamiltonian} to determine the threshold for the Hamiltonicity of random graphs. P\'osa's proof is notable for being the first to consider \textit{the family of all possible sequences of rotations} that can be performed to a given path of maximal length. This allows one to show that many vertices can become an endpoint of such a path, which in turn is crucial for showing that sufficiently dense random graphs are Hamiltonian. Since this was a very influential idea, rotations are often referred to as \emph{P\'osa rotations} in the literature. Although the idea of rotations has been around for a long time, several breakthroughs in recent years were only made possible by introducing novel twists to this method \cite{glock2024hamilton, draganic2024pancyclicity, draganic2024hamiltonicity}.

\begin{figure}[ht]
    \centering
    \begin{tikzpicture}[line width=1pt]
    \begin{scope}
        \begin{scope}[graph]
            \node[label=below:$v$] (v) at (0,0) {};
            \node[label=below:$w$] (w) at (2,0) {};
            \node[label=below:$u$] (u) at (3,0) {};
            \draw (v) -- (w) -- (u) -- (4,0);
            \draw[dotted] (v) to[bend left=50] (u);
        \end{scope}
        \draw[->] (5.1,0) -- (5.9,0);
        \begin{scope}[graph,xshift=7cm]
            \node[label=below:$v$] (v) at (0,0) {};
            \node[label=below:$w$] (w) at (2,0) {};
            \node[label=below:$u$] (u) at (3,0) {};
            \draw (v) -- (w)  (u) -- (4,0) (v) to[bend left=50] (u);
            \draw[dotted] (w) -- (u);
        \end{scope}
    \end{scope}
    \begin{scope}[yshift=-3cm]
        \begin{scope}[graph]
            \node[label=below:$v$] (v) at (0,0) {};
            \node[label=above:$u$] (u) at (1.5,1) {};
            \node[label=above:$w$] (w) at (2.5,1) {};
            \draw (v) -- (4,0) (0,1) -- (u) -- (w) -- (4,1);
            \draw[dotted] (v) -- (u);
        \end{scope}
        \draw[->] (5.1,0.5) -- (5.9,0.5);
        \begin{scope}[graph,xshift=7cm]
            \node[label=below:$v$] (v) at (0,0) {};
            \node[label=above:$u$] (u) at (1.5,1) {};
            \node[label=above:$w$] (w) at (2.5,1) {};
            \draw (v) -- (4,0) (0,1) -- (u) -- (v) (w) -- (4,1);
            \draw[dotted] (u) -- (w);
        \end{scope}
    \end{scope}
\end{tikzpicture}
    
    \caption{The two different types of rotations of a linear forest. Dotted edges represent edges of the graph that are not contained in the linear forest.}
    \label{fig:linearforestrotations}
\end{figure}
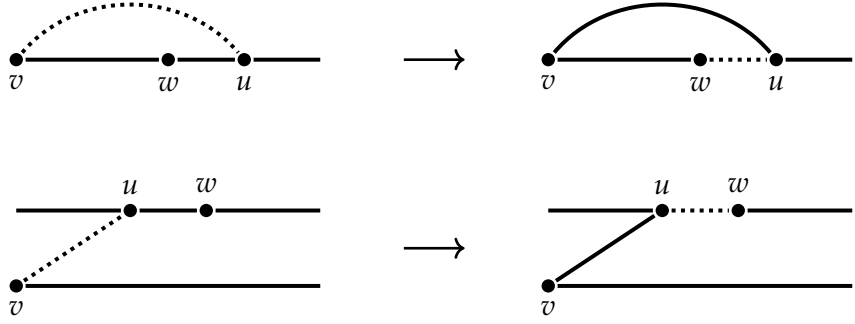

A common feature of these applications is some notion of expansion in the underlying graph. In contrast, we apply rotations to sparse regular graphs without any expansion conditions. Moreover, while P\'{o}sa rotations are only about rotations of \emph{a single endpoint} of a path, our main novelty is to consider simultaneous rotations of \emph{a set of endpoints} of a linear forest. Concretely, our key lemma, \cref{prop:componentsize}, proves the following.
\begin{quote}
    Suppose we are given a minimum spanning linear forest of a $d$-regular graph and a set of $\ell$ endpoints. If we allow any sequence of rotations that only involves this set of endpoints, then at least $\Omega(\ell\cdot d)$ vertices can become an endpoint.
\end{quote}
As we make no assumptions on the structure of the underlying regular graph, this requires a highly delicate argument. We prove this lemma in \cref{sec:smalllinearforests} where we also concisely explain the key idea of its proof. For now, note that if there are at least $\Omega(\ell \cdot d)$ distinct potential endpoints, it immediately follows that $\ell = \cO(n/d)$, which implies \cref{thm:smalllinearforest} and the conjecture of Feige and Fuchs. Moreover, in \cref{sec:lineararboricity}, we use this key lemma to create a distribution on the spanning linear forests of the graph in which every vertex has a probability of at most $\cO(1/d)$ of being an endpoint. To prove our bounds on the (fractional) linear arboricity, \cref{thm:lineararboricitylogerror,thm:fractionallineararboricitylogerror}, we then iteratively remove $\Delta/2$ such random linear forests from the graph and decompose the remaining edges into matchings using Vizing's theorem. Note that a vertex has a high degree in the remaining graph only if it was an endpoint in many of the removed linear forests. Since every vertex is unlikely to be an endpoint, we can apply standard concentration inequalities and a union bound to show that the maximum degree of the remaining graph is $\cO(\log n)$. This gives the error term of $\cO(\log n)$ in our result. We refer the reader to \cref{sec:openproblems} for a discussion on the difficulties of improving this error term further.

Throughout the remainder of the paper, every linear forest is a spanning linear forest.

\section{Small linear forests in regular graphs}\label{sec:smalllinearforests}

In this section, we will show that any $d$-regular graph $G$ on $n$ vertices has a spanning linear forest with at most $2 \cdot n/(d+1)$ paths. Our main tool to achieve this are rotations of linear forests. A rotation modifies a linear forest by adding an edge incident to some endpoint of the linear forest and removing a different edge so that the resulting subgraph is still a linear forest. Intuitively, a rotation moves one endpoint of the linear forest to a different vertex of the graph. \cref{fig:linearforestrotations} illustrates the two different types of possible rotations of a linear forest.

\begin{definition}
    Let $F$ be a linear forest of a graph $G$. Let $v$ be an endpoint of $F$, and let $u$ be a neighbour of $v$ in $G$. If $u$ and $v$ are on the same path in $F$, let $w$ be the neighbour of $u$ that is closer to $v$ on that path. Otherwise, let $w$ be any neighbour of $u$ on its path in $F$. Then, the linear forest $F'$ obtained from $F$ by removing the edge $u w$ and adding the edge $u v$ is called a \emph{rotation} of $F$ in $G$.
\end{definition}

We say that such a rotation rotates the \emph{old endpoint} $v$ to the \emph{new endpoint} $w$, using the \emph{pivot} $u$. The restriction on $w$ in the case where $u$ and $v$ are on the same path is necessary to ensure that the resulting subgraph is still a linear forest. Note that $w$ is allowed to be the same vertex as $v$, which could happen if $u$ is the unique neighbour of $v$ on its path in $F$. Also, observe that if $F'$ is a rotation of $F$, then $F'$ has the same number of paths as $F$.

If we can use rotations to reach a linear forest where two endpoints of \emph{distinct} paths are at adjacent vertices of the graph, then we can reduce the number of paths of the linear forest by adding the edge between these two endpoints. Suppose that $F$ is a linear forest where we can no longer do this, for example because $F$ has the minimum number of paths possible. To show that $F$ has few paths, we will consider the set $C$ of all vertices of the graph that could become an endpoint of $F$ through some sequence of rotations. We claim that if $F$ has $\ell$ endpoints, then the size of $C$ is at least $\ell \cdot (d+1)/4$. Since the size of $C$ is at most the total number of vertices of the graph, this claim implies that $F$ has at most $4 \cdot n / (d+1)$ endpoints and therefore at most $2 \cdot n / (d+1)$ paths, as required.

We now sketch our proof of this claim. First, suppose for simplicity that $G$ is bipartite. In this case, by only rotating endpoints in one part of the bipartition of $G$, we may assume that $C$ is an independent set. Then, we will analyse the edges between $C$ and its \emph{boundary} $B \coloneqq N_G(C) \setminus C$\footnote{If $G$ is a graph and $X \subseteq V(G)$, the \emph{neighbourhood} of $X$ in $G$ is $N_G(X) \coloneqq \{v \in V(G) : u v \in E(G) \text{ for some } u \in X\}$.}.

In particular, we exploit the tension between the following two facts. Firstly, each vertex in $C$ has $d$ neighbours in $B$ while each vertex in $B$ has at most $d$ neighbours in $C$. Secondly, each vertex in $C$ has at most two path-neighbours\footnote{A \emph{path-neighbour} of a vertex $v$ is a neighbour of $v$ on its path in $F$.} in $B$ while each vertex in $B$ has at least one path-neighbour in $C$\footnote{The fact that every vertex in $B$ has at least one path-neighbour in $C$ follows from a standard argument for P\'{o}sa rotations.}. Using a double counting argument on these edges and combining this with the fact that $C$ contains at least one endpoint, we can find a vertex $u \in B$ with $(d+1)/2$ neighbours in $C$, but which has only one path-neighbour in $C$.

Now, rotate one endpoint of $F$ to an arbitrary neighbour $v$ of $u$ and \emph{fix it there}. If we never allow $v$ to be rotated again, we claim that no other endpoint can be rotated to a neighbour $v'$ of $u$. Indeed, if that were possible, a short case analysis shows that either one of the two endpoints $v$ or $v'$ could be rotated to a vertex outside of $C$, contradicting the definition of $C$, or one of the two edges $u v$ or $u v'$ could be added to $F$, contradicting the minimality of $F$. This case analysis uses the fact that $u$ has only one path-neighbour in $C$.

Therefore, fixing an endpoint of $F$ in the neighbourhood of $u$ reduces the number of endpoints that we rotate by one while it reduces the number of vertices that could become an endpoint through a sequence of rotations by at least $(d+1)/2$. Iterating this argument shows that $C$ has size at least $\ell \cdot (d+1)/2$, proving the result for bipartite graphs.

For non-bipartite graphs, we divide the argument into two cases based on the number of edges spanned by $C$. If $C$ spans few edges, we can proceed exactly as for bipartite graphs. If instead $C$ spans many edges, we iteratively fix endpoints at vertices that have a high degree within $C$ and show that this reduces the size of $C$ sufficiently. In both cases, we then iterate the argument to show that $C$ has size at least $\ell \cdot (d+1)/4$, as required.

We now formalise this strategy. We begin by defining the set of all vertices of the graph that could become an endpoint of a linear forest through a sequence of rotations. In the following, we will use $\End(F)$ to denote the multiset of endpoints of a linear forest $F$, where an isolated vertex of $F$ is contained twice in $\End(F)$. In particular, if $F'$ is a rotation of $F$ that rotates $v$ to $w$, then $\End(F') = (\End(F) \setminus \{v\}) \cup \{w\}$.

\begin{definition}
    Let $F$ be a linear forest of a graph $G$, and let $X \subseteq \End(F)$ be a multiset of \emph{fixed endpoints}. Then, $\cF(F,X)$ denotes the family of linear forests obtained by a sequence of rotations starting at $F$ that never rotates any endpoint of $X$. Define the set
    \[
        C(F,X) \coloneqq \bigcup_{F' \in \cF(F,X)} \End(F') \setminus X.
    \]
\end{definition}

We call $C(F,X)$ a \emph{rotation-component}. Intuitively, $C(F,X)$ consists of all vertices that could become an endpoint of $F$ through a sequence of rotations that leaves all endpoints of $X$ fixed. Note that if an endpoint $v$ is contained once in $X$ but twice in $\End(F)$, then we allow one of the endpoints at $v$ to be rotated, but not both. In particular, $C(F,X)$ might not be disjoint from $X$. Since rotations are reversible, we have $\cF(F,X) = \cF(F',X)$ for all $F' \in \cF(F,X)$, and therefore also $C(F,X) = C(F',X)$. Moreover, if $X \subseteq Y \subseteq \End(F)$, then $C(F,X) \supseteq C(F,Y)$.

A \emph{minimum linear forest} $F$ of $G$ is a linear forest with the minimum number of paths among all linear forests of $G$. We now state the key lemma of this section which asserts that for any minimum linear forest $F$, the size of the rotation-component $C(F,X)$ is always at least $(d+1)/4$ times the number of endpoints that are not fixed.

\begin{lemma}\label{prop:componentsize}
    Let $F$ be a minimum linear forest of a $d$-regular graph, and let $X \subseteq \End(F)$. Then,
    \[
        \abs{C(F,X)} \ge \frac{d+1}{4} \cdot \abs{\End(F) \setminus X}.
    \]
\end{lemma}

This result suffices to prove that every $d$-regular graph on $n$ vertices has a linear forest with at most $2 \cdot n / (d+1)$ paths.

\begin{proof}[Proof of \cref{thm:smalllinearforest}]
    Let $G$ be a $d$-regular graph on $n$ vertices, and let $F$ be a minimum linear forest of $G$. By \cref{prop:componentsize}, we have
    \[
        \abs{\End(F)} \le \frac{4}{d+1} \cdot \abs{C(F,\emptyset)} \le \frac{4}{d+1} \cdot n.
    \]
    Since the number of endpoints of a linear forest is twice the number of its paths, this implies that $F$ has at most $2 \cdot n/(d+1)$ paths.
\end{proof}

To prove \cref{prop:componentsize}, we first make the following observations about rotations that are used extensively throughout our proofs. These observations follow directly from the definition of rotations together with a small amount of case analysis as indicated in \cref{fig:rotationobservations}.
\begin{enumerate}[(R1)]
    \item Suppose that $v$ and $u$ are endpoints of $F$ and neighbours in $G$.
    \begin{enumerate}[(a),ref=(R1\alph*)]
        \item\label{obs:rotateonetoone} If $w$ is a neighbour of $u$ on its path in $F$, then we can rotate $v$ to $w$ using the pivot $u$.
        \item\label{obs:reduceonetoone} Alternatively, if $u$ and $v$ are on different paths in $F$, we can reduce the number of paths of $F$ by adding the edge $u v$ to $F$. In particular, this applies if $u$ is an isolated vertex of $F$.
    \end{enumerate}
    \item\label{obs:rotateonetotwo} Suppose that $v$ is an endpoint of $F$ which is a neighbour of $u$ in $G$. If $w$ and $w'$ are neighbours of $u$ on its path in $F$, then we can rotate $v$ to one of the two vertices $w$ or $w'$ using the pivot $u$.
    \item Suppose that $v$ and $v'$ are endpoints of $F$ which are both neighbours of $u$ in $G$. Note that we allow $v = v'$ if $v$ is an isolated vertex of $F$.
    \begin{enumerate}[(a),ref=(R3\alph*)]
        \item\label{obs:rotatetwotoone} If $w$ is a neighbour of $u$ on its path in $F$, then we can rotate one of the two endpoints $v$ or $v'$ to $w$ using the pivot $u$.
        \item\label{obs:reducetwotoone} Alternatively, if $u$ is an endpoint of $F$, we can reduce the number of paths of $F$ by adding one of the two edges $u v$ or $u v'$ to $F$.
    \end{enumerate}
\end{enumerate}

\begin{figure}[t]
    \centering
    \begin{tikzpicture}[scale=0.8,line width=1pt]
    \draw (3,2.5) -- (3,-15);
    \draw (7,2.5) -- (7,-15);
    \draw (11,2.5) -- (11,-15);
    \draw (15,2.5) -- (15,-15);
    \begin{scope}
        \node at (1,2) {(R1a)};
        \begin{scope}[graph]
            \node[label=below:$v$] (v) at (0,0) {};
            \node[label=below:$u$] (u) at (1,0) {};
            \node[label=below:$w$] (w) at (2,0) {};
            \draw (u) -- (w);
            \draw[dotted] (v) -- (u);
            \draw[snake] (-0.4,0.9) -- (v);
            \draw[snake] (2.4,0.9) -- (w);
        \end{scope}
        \node at (1,-1.5) {or};
        \begin{scope}[graph,yshift=-3.5cm]
            \node[label=below:$v$] (v) at (0,0) {};
            \node[label=below:$u$] (u) at (1,0) {};
            \node[label=below:$w$] (w) at (2,0) {};
            \draw (v) (u) -- (w);
            \draw[dotted] (v) -- (u);
            \draw[snake] (v) to[bend left=80,looseness=1.6] (w);
        \end{scope}
        \draw[->] (1,-5.1) -- (1,-5.9);
        \begin{scope}[graph,yshift=-7cm]
            \node[label=below:$v$] (v) at (0,0) {};
            \node[label=below:$u$] (u) at (1,0) {};
            \node[label=below:$w$] (w) at (2,0) {};
            \draw (v) -- (u);
            \draw[dotted] (u) -- (w);
        \end{scope}
    \end{scope}
    
    \begin{scope}[xshift=4cm]
        \node at (1,2) {(R1b)};
        \begin{scope}[graph,xshift=0.5cm]
            \node[label=below:$v$] (v) at (0,0) {};
            \node[label=below:$u$] (u) at (1,0) {};
            \draw[dotted] (v) -- (u);
            \draw[snake] (-0.4,0.9) -- (v);
            \draw[snake] (1.4,0.9) -- (u);
        \end{scope}
        \draw[->] (1,-1.6) -- (1,-2.4);
        \begin{scope}[graph,xshift=0.5cm,yshift=-3.5cm]
            \node[label=below:$v$] (v) at (0,0) {};
            \node[label=below:$u$] (u) at (1,0) {};
            \draw (v) -- (u);
        \end{scope}
    \end{scope}
    
    \begin{scope}[xshift=8cm]
        \node at (1,2) {(R2)};
        \draw[dotted] (-0.5,-8.5) -- (2.5,-8.5);
        \begin{scope}[graph]
            \node[label=below:$v$] (v) at (0,0) {};
            \node[label=below:$u$] (u) at (1,0) {};
            \node[label=right:$w$] (w) at (1.8,0.5) {};
            \node[label=right:$w'$] (wp) at (1.8,-0.5) {};
            \draw (w) -- (u) -- (wp);
            \draw[dotted] (v) -- (u);
            \draw[snake] (-0.4,0.9) -- (v);
            \draw[snake] (2.2,1.2) -- (w);
            \draw[snake] (2.2,-1.2) -- (wp);
        \end{scope}
        \node at (1,-1.5) {or};
        \begin{scope}[graph,yshift=-3.5cm]
            \node[label=below:$v$] (v) at (0,0) {};
            \node[label=below:$u$] (u) at (1,0) {};
            \node[label=right:$w$] (w) at (1.8,0.5) {};
            \node[label=right:$w'$] (wp) at (1.8,-0.5) {};
            \draw (w) -- (u) -- (wp);
            \draw[dotted] (v) -- (u);
            \draw[snake] (v) to[bend left=70,looseness=1.4] (w);
            \draw[snake] (2.2,-1.2) -- (wp);
        \end{scope}
        \draw[->] (1,-5.1) -- (1,-5.9);
        \begin{scope}[graph,yshift=-7cm]
            \node[label=below:$v$] (v) at (0,0) {};
            \node[label=below:$u$] (u) at (1,0) {};
            \node[label=right:$w$] (w) at (1.8,0.5) {};
            \node[label=right:$w'$] (wp) at (1.8,-0.5) {};
            \draw (v) -- (u) -- (wp);
            \draw[dotted] (u) -- (w);
        \end{scope}
    
        \begin{scope}[yshift=-10.5cm]
            \begin{scope}[graph]
                \node[label=above:$v$] (v) at (0,0) {};
                \node[label=above:$u$] (u) at (1,0) {};
                \node[label=right:$w$] (w) at (1.8,0.5) {};
                \node[label=right:$w'$] (wp) at (1.8,-0.5) {};
                \draw (w) -- (u) -- (wp);
                \draw[dotted] (v) -- (u);
                \draw[snake] (2.2,1.2) -- (w);
                \draw[snake] (v) to[bend right=70,looseness=1.4] (wp);
            \end{scope}
            \draw[->] (1,-1.6) -- (1,-2.4);
            \begin{scope}[graph,yshift=-3.5cm]
                \node[label=below:$v$] (v) at (0,0) {};
                \node[label=below:$u$] (u) at (1,0) {};
                \node[label=right:$w$] (w) at (1.8,0.5) {};
                \node[label=right:$w'$] (wp) at (1.8,-0.5) {};
                \draw (v) -- (u) -- (w);
                \draw[dotted] (u) -- (wp);
            \end{scope}
        \end{scope}
    \end{scope}
    
    \begin{scope}[xshift=12cm]
        \node at (1,2) {(R3a)};
        \draw[dotted] (-0.5,-8.5) -- (2.5,-8.5);
        \begin{scope}[graph]
            \node[label=left:$v$] (v) at (0.2,0.5) {};
            \node[label=left:$v'$] (vp) at (0.2,-0.5) {};
            \node[label=below:$u$] (u) at (1,0) {};
            \node[label=below:$w$] (w) at (2,0) {};
            \draw (u) -- (w);
            \draw[dotted] (v) -- (u) -- (vp);
            \draw[snake] (-0.2,1.2) -- (v);
            \draw[snake] (-0.2,-1.2) -- (vp);
            \draw[snake] (1.4,0.9) -- (u);
            \draw[snake] (2.4,0.9) -- (w);
        \end{scope}
        \node at (1,-1.5) {or};
        \begin{scope}[graph,yshift=-3.5cm]
            \node[label=left:$v$] (v) at (0.2,0.5) {};
            \node[label=left:$v'$] (vp) at (0.2,-0.5) {};
            \node[label=above:$u$] (u) at (1,0) {};
            \node[label=above:$w$] (w) at (2,0) {};
            \draw (u) -- (w);
            \draw[dotted] (v) -- (u) -- (vp);
            \draw[snake] (-0.2,1.2) -- (v);
            \draw[snake] (vp) to[bend right=100,looseness=2.5] (u);
            \draw[snake] (2.4,-0.9) -- (w);
        \end{scope}
        \draw[->] (1,-5.1) -- (1,-5.9);
        \begin{scope}[graph,yshift=-7cm]
            \node[label=left:$v$] (v) at (0.2,0.5) {};
            \node[label=left:$v'$] (vp) at (0.2,-0.5) {};
            \node[label=below:$u$] (u) at (1,0) {};
            \node[label=below:$w$] (w) at (2,0) {};
            \draw (v) -- (u);
            \draw[dotted] (vp) -- (u) -- (w);
        \end{scope}
    
        \begin{scope}[yshift=-10.5cm]
            \begin{scope}[graph]
                \node[label=left:$v$] (v) at (0.2,0.5) {};
                \node[label=left:$v'$] (vp) at (0.2,-0.5) {};
                \node[label=below:$u$] (u) at (1,0) {};
                \node[label=below:$w$] (w) at (2,0) {};
                \draw (u) -- (w);
                \draw[dotted] (v) -- (u) -- (vp);
                \draw[snake] (v) to[bend left=100,looseness=2.5] (u);
                \draw[snake] (-0.2,-1.2) -- (vp);
                \draw[snake] (2.4,0.9) -- (w);
            \end{scope}
            \draw[->] (1,-1.6) -- (1,-2.4);
            \begin{scope}[graph,yshift=-3.5cm]
                \node[label=left:$v$] (v) at (0.2,0.5) {};
                \node[label=left:$v'$] (vp) at (0.2,-0.5) {};
                \node[label=below:$u$] (u) at (1,0) {};
                \node[label=below:$w$] (w) at (2,0) {};
                \draw (vp) -- (u);
                \draw[dotted] (v) -- (u) -- (w);
            \end{scope}
        \end{scope}
    \end{scope}
    
    \begin{scope}[xshift=16cm]
        \node at (1,2) {(R3b)};
        \draw[dotted] (-0.5,-8.5) -- (2.5,-8.5);
        \begin{scope}[graph,xshift=0.5cm]
            \node[label=left:$v$] (v) at (0.2,0.5) {};
            \node[label=left:$v'$] (vp) at (0.2,-0.5) {};
            \node[label=below:$u$] (u) at (1,0) {};
            \draw[dotted] (v) -- (u) -- (vp);
            \draw[snake] (-0.2,1.2) -- (v);
            \draw[snake] (-0.2,-1.2) -- (vp);
            \draw[snake] (2,0) -- (u);
        \end{scope}
        \node at (1,-1.5) {or};
        \begin{scope}[graph,xshift=0.5cm,yshift=-3.5cm]
            \node[label=left:$v$] (v) at (0.2,0.5) {};
            \node[label=left:$v'$] (vp) at (0.2,-0.5) {};
            \node[label=above:$u$] (u) at (1,0) {};
            \draw[dotted] (v) -- (u) -- (vp);
            \draw[snake] (-0.2,1.2) -- (v);
            \draw[snake] (vp) to[bend right=100,looseness=2.5] (u);
        \end{scope}
        \draw[->] (1,-5.1) -- (1,-5.9);
        \begin{scope}[graph,xshift=0.5cm,yshift=-7cm]
            \node[label=left:$v$] (v) at (0.2,0.5) {};
            \node[label=left:$v'$] (vp) at (0.2,-0.5) {};
            \node[label=below:$u$] (u) at (1,0) {};
            \draw (v) -- (u);
            \draw[dotted] (vp) -- (u);
        \end{scope}
    
        \begin{scope}[yshift=-10.5cm]
            \begin{scope}[graph,xshift=0.5cm]
                \node[label=left:$v$] (v) at (0.2,0.5) {};
                \node[label=left:$v'$] (vp) at (0.2,-0.5) {};
                \node[label=below:$u$] (u) at (1,0) {};
                \draw[dotted] (v) -- (u) -- (vp);
                \draw[snake] (v) to[bend left=100,looseness=2.5] (u);
                \draw[snake] (-0.2,-1.2) -- (vp);
            \end{scope}
            \draw[->] (1,-1.6) -- (1,-2.4);
            \begin{scope}[graph,xshift=0.5cm,yshift=-3.5cm]
                \node[label=left:$v$] (v) at (0.2,0.5) {};
                \node[label=left:$v'$] (vp) at (0.2,-0.5) {};
                \node[label=below:$u$] (u) at (1,0) {};
                \draw (vp) -- (u);
                \draw[dotted] (v) -- (u);
            \end{scope}
        \end{scope}
    \end{scope}
\end{tikzpicture}
    \caption{The case analysis for the observations about possible rotations of linear forests.}
    \label{fig:rotationobservations}
\end{figure}
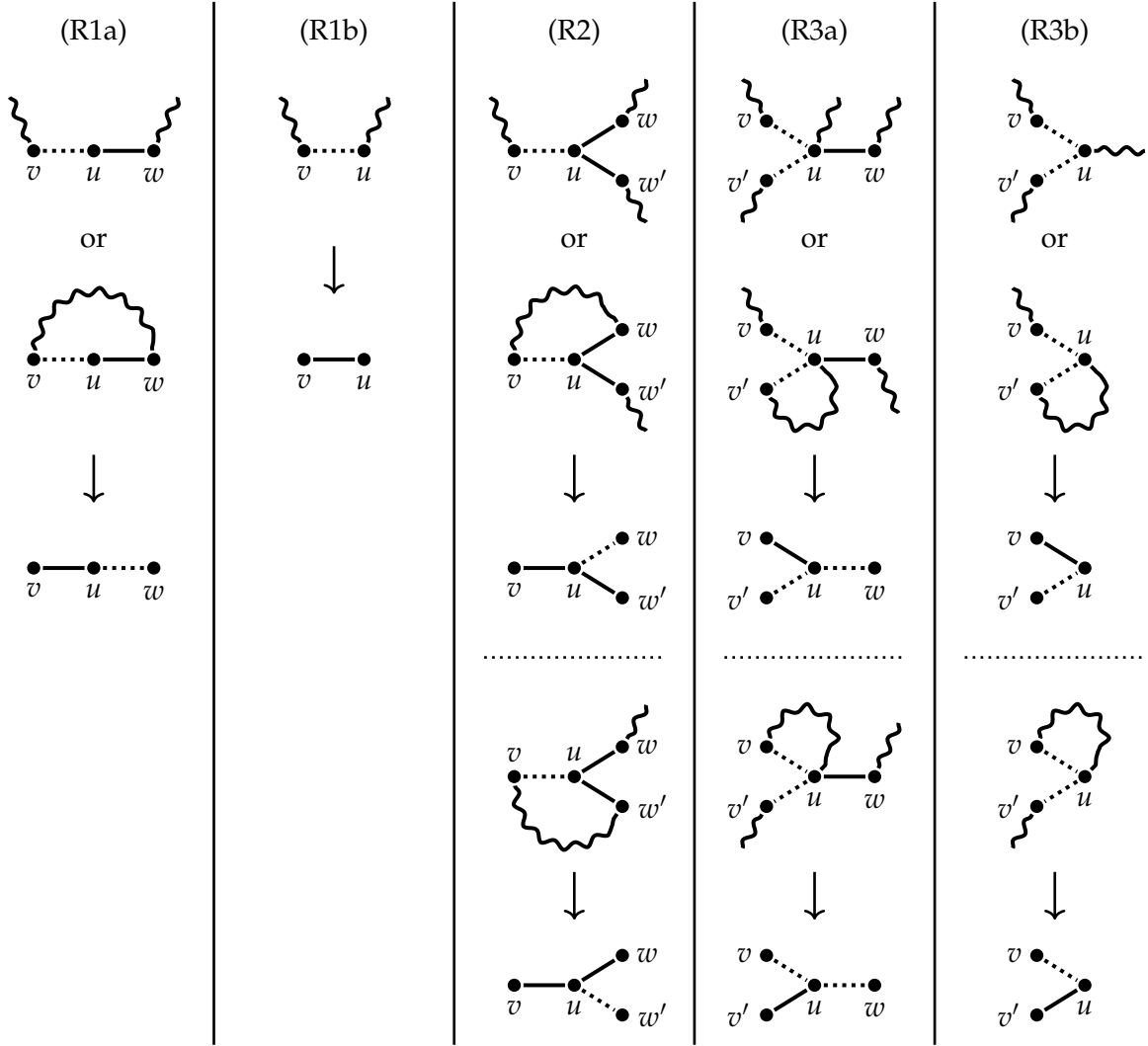

Now, consider the \emph{boundary} $B(F,X) \coloneqq N_G(C(F,X)) \setminus C(F,X)$ of $C(F,X)$. Since $C(F,X) = C(F',X)$ for all $F' \in \cF(F,X)$, we also have $B(F,X) = B(F',X)$ for all $F' \in \cF(F,X)$. It is useful to classify each vertex of $C(F,X)$ and $B(F,X)$ based on how many neighbours on its path in $F$ are contained in the opposite set. So, define $C_i(F,X) \coloneqq \{v \in C(F,X) : \abs{N_F(v) \cap B(F,X)} = i\}$ and $B_i(F,X) \coloneqq \{v \in B(F,X) : \abs{N_F(v) \cap C(F,X)} = i\}$. Since every vertex is incident to at most two edges of $F$, we may assume that $i \in \{0, 1, 2\}$.

We start by proving some basic properties about these sets. The second of these properties, namely that $B_0(F,X) = \emptyset$, is a well-known fact for P\'{o}sa rotations.

\begin{lemma}\label{lem:b0empty}
    Let $F$ be a linear forest and let $X \subseteq \End(F)$. Then, for all $F' \in \cF(F,X)$ and $i \in \{0, 1, 2\}$ we have $B_i(F,X) = B_i(F',X)$. Moreover, if $F$ is a minimum linear forest, then $B_0(F,X) = \emptyset$.
\end{lemma}

\begin{proof}
    Note that $F'$ is obtained by a sequence of rotations starting at $F$ that never rotates any endpoint of $X$. It suffices to consider a single rotation of this sequence, so we may assume that $F'$ is a rotation of $F$ that rotates an old endpoint $v$ to a new endpoint $w$, using a pivot $u$. Note that $v, w \in C(F,X) = C(F',X)$. In particular, $\abs{N_F(u) \cap C(F,X)} = \abs{N_{F'}(u) \cap C(F',X)}$, and for each vertex in $B(F,X)$ apart from $u$, the edges of $F$ and $F'$ incident to that vertex are exactly the same. This shows that $B_i(F,X) = B_i(F',X)$.

    Now, let $F$ be a minimum linear forest. We want to show that $B_0(F,X) = \emptyset$, so suppose for a contradiction that there exists a vertex $u \in B_0(F,X)$. Let $v \in N_G(u) \cap C(F,X)$ be any neighbour of $u$ in $C(F,X)$. Since $v \in C(F,X)$, there exists a linear forest $F' \in \cF(F,X)$ with $X \cup \{v\} \subseteq \End(F')$. Because $F'$ has the same number of paths as $F$, $F'$ is a minimum linear forest, and we know that $u \in B_0(F,X) = B_0(F',X)$.

    If $u$ is an isolated vertex of $F'$, then by \labelcref{obs:reduceonetoone} we could reduce the number of paths of $F'$ by adding the edge $u v$, contradicting that $F'$ is a minimum linear forest. If $u$ is an endpoint of $F'$ but has a neighbour $w \in N_{F'}(u)$ on its path in $F'$, then by \labelcref{obs:rotateonetoone} we could rotate $v$ to $w$ using the pivot $u$. But this would imply that $w \in C(F',X)$, contradicting $u \in B_0(F',X)$. Finally, if $u$ has two neighbours on its path in $F'$, then by \labelcref{obs:rotateonetotwo} there exists at least one neighbour $w \in N_{F'}(u)$ such that we could rotate $v$ to $w$ using the pivot $u$, again contradicting $w \notin C(F',X)$. In all cases we get a contradiction, and so $B_0(F,X) = \emptyset$.
\end{proof}

We now show that we can reduce the size of $C(F,X)$ by fixing an endpoint at a vertex in $C(F,X)$ that has a neighbour in $B_1(F,X)$ which itself has many neighbours in $C(F,X)$.

\begin{lemma}\label{lem:fixneighbourofb1}
    Let $F$ be a minimum linear forest and let $X \subseteq \End(F)$. Let $v \in \End(F) \setminus X$ and suppose that $u \in N_G(v) \cap B_1(F,X)$. Then, $\abs{C(F,X \cup \{v\})} \le \abs{C(F,X)} - \abs{N_G(u) \cap C(F,X)}$.
\end{lemma}

\begin{proof}
    We claim that $N_G(u) \cap C(F,X \cup \{v\}) = \emptyset$. Suppose for a contradiction that there exists a vertex $v' \in N_G(u) \cap C(F,X \cup \{v\})$. Since $v' \in C(F,X \cup \{v\})$, there exists a linear forest $F' \in \cF(F,X \cup \{v\})$ with $X \cup \{v, v'\} \subseteq \End(F')$. Note that $F'$ is a minimum linear forest. By \cref{lem:b0empty} we know that $u \in B_1(F,X) = B_1(F',X)$.

    If $u$ is an endpoint of $F'$, then by \labelcref{obs:reducetwotoone} we could reduce the number of paths of $F'$ by adding one of the two edges $u v$ or $u v'$, contradicting that $F'$ is a minimum linear forest. Otherwise, $u$ has two neighbours on its path in $F'$. Since $u \in B_1(F',X)$, there must be a neighbour $w \in N_{F'}(u)$ that is not contained in $C(F',X)$. But then by \labelcref{obs:rotatetwotoone} we could rotate one of the two endpoints $v$ or $v'$ to $w$ using the pivot $u$, contradicting $w \notin C(F',X)$. In both cases we get a contradiction, and so $N_G(u) \cap C(F,X \cup \{v\}) = \emptyset$, as claimed. Since $C(F,X \cup \{v\}) \subseteq C(F,X)$, this implies that $\abs{C(F,X \cup \{v\})} \le \abs{C(F,X)} - \abs{N_G(u) \cap C(F,X)}$.
\end{proof}

To use this result to prove a good lower bound on $\abs{C(F,X)}$, we need to find a vertex in $B_1(F,X)$ with many neighbours in $C(F,X)$. This is easy if there are many edges between $C(F,X)$ and $B_1(F,X)$. We will use the following double counting argument to provide a lower bound on the number of such edges. Here, for $A, B \subseteq V(G)$, we denote the number of edges of $G$ connecting a vertex in $A$ to a vertex in $B$ by $e(A,B)$, and we write $e(A) \coloneqq e(A,A)$.

\begin{lemma}\label{lem:doublecount}
    Let $F$ be a minimum linear forest of a $d$-regular graph $G$, and let $X \subseteq \End(F)$. Then,
    \[
        e(C(F,X),B_1(F,X)) \ge \frac{d}{2} \cdot \abs{B_1(F,X)} + \frac{d}{2} \cdot \abs{C_1(F,X)} + d \cdot \abs{C_0(F,X)} - 2 \cdot e(C(F,X)).
    \]
\end{lemma}

\begin{proof}
    By double counting the number of edges of $F$ between $C(F,X)$ and $B(F,X)$, we get
    \[
        \abs{C_1(F,X)} + 2 \cdot \abs{C_2(F,X)} = \abs{B_1(F,X)} + 2 \cdot \abs{B_2(F,X)}.
    \]
    Note that the left side of this equation is the same as $2 \cdot \abs{C(F,X)} - \abs{C_1(F,X)} - 2 \cdot \abs{C_0(F,X)}$. The total number of edges of $G$ between $C(F,X)$ and $B(F,X)$ is $d \cdot \abs{C(F,X)} - 2 \cdot e(C(F,X))$. Because at most $d \cdot \abs{B_2(F,X)}$ of these edges are incident with $B_2(F,X)$ and we know that $B_0(F,X) = \emptyset$ by \cref{lem:b0empty}, it follows that
    \begin{align*}
        e(C(F,X),B_1(F,X)) & \ge d \cdot \abs{C(F,X)} - 2 \cdot e(C(F,X)) - d \cdot \abs{B_2(F,X)} \\
        & = \frac{d}{2} \cdot \abs{B_1(F,X)} + \frac{d}{2} \cdot \abs{C_1(F,X)} + d \cdot \abs{C_0(F,X)} - 2 \cdot e(C(F,X)). \qedhere
    \end{align*}
\end{proof}

We now first prove that bipartite graphs have small linear forests. In that setting, observe that if we fix all endpoints that are contained in one part of the bipartition, then $C(F,X)$ will be contained in the other part of the bipartition. This implies that $e(C(F,X)) = 0$. As a result, \cref{lem:doublecount} shows that there exists a vertex in $B_1(F,X)$ with at least $(d+1)/2$ neighbours in $C(F,X)$. This allows us to apply \cref{lem:fixneighbourofb1} and use induction to complete the proof.

\begin{theorem}\label{thm:smalllinearforestbipartite}
    Every $d$-regular $n$-vertex bipartite graph has a spanning linear forest with at most $n / (d+1)$ paths.
\end{theorem}

\begin{proof}
    Let $G$ be a $d$-regular bipartite graph on $n$ vertices, and let $U$ be one part of its bipartition. We claim that if $F$ is a minimum linear forest of $G$ and $X \subseteq \End(F)$ satisfies $\End(F) \setminus X \subseteq U$, then $C(F,X) \subseteq U$ and $\abs{C(F,X)} \ge \abs{\End(F) \setminus X} \cdot (d+1)/2$. We prove the claim by induction on $\ell \coloneqq \abs{\End(F) \setminus X}$. The claim clearly holds for $\ell = 0$, so suppose that $\ell > 0$.

    Observe that if a rotation of $F$ rotates an old endpoint $v \in U$ to a new endpoint $w$, then $w$ is at distance two from $v$. Since $G$ is bipartite, it follows that $w \in U$. This shows that $C(F,X) \subseteq U$. In particular, $e(C(F,X)) = 0$. Since $\emptyset \neq \End(F) \setminus X \subseteq C_1(F,X) \cup C_0(F,X)$, this implies by \cref{lem:doublecount} that $e(C(F,X),B_1(F,X)) > \abs{B_1(F,X)} \cdot d/2$. In particular, there exists a vertex $u \in B_1(F,X)$ with $\abs{N_G(u) \cap C(F,X)} \ge (d+1)/2$.

    Let $v \in N_G(u) \cap C(F,X)$ be arbitrary, let $X' \coloneqq X \cup \{v\}$, and let $F' \in \cF(F,X)$ be any linear forest with $X' \subseteq \End(F')$. Using the fact that $C(F,X) = C(F',X)$, \cref{lem:fixneighbourofb1} implies that
    \[
        \abs{C(F,X)} = \abs{C(F',X)} \ge \abs{C(F',X')} + \abs{N_G(u) \cap C(F',X)} \ge \abs{C(F',X')} + \frac{d+1}{2}.
    \]
    By the induction hypothesis, it then follows that
    \[
        \abs{C(F,X)} \ge \abs{C(F',X')} + \frac{d+1}{2} \ge \frac{d+1}{2} \cdot \abs{\End(F') \setminus X'} + \frac{d+1}{2} = \frac{d+1}{2} \cdot \abs{\End(F) \setminus X}.
    \]
    This proves the claim. Finally, if $V(G) = V \cup W$ is the bipartition of $G$, the claim implies that
    \begin{align*}
        \abs{\End(F)} & = \abs{\End(F) \setminus (\End(F) \cap V)} + \abs{\End(F) \setminus (\End(F) \cap W)} \\
        & \le \frac{2}{d+1} \cdot \abs{C(F,\End(F) \cap V)} + \frac{2}{d+1} \cdot \abs{C(F,\End(F) \cap W)} \le \frac{2}{d+1} \cdot n.
    \end{align*}
    Since the number of endpoints of a linear forest is twice the number of its paths, this implies that $F$ has at most $n/(d+1)$ paths.
\end{proof}

In non-bipartite graphs, we will define an iterative process that fixes endpoints of $F$ one by one at vertices that have a high degree within $C(F,X)$. If this reduces the size of $C(F,X)$ sufficiently, we are done by induction. Otherwise, we show that $C(F,X)$ spans only few edges, and so we can finish the proof as in the bipartite case.

\begin{proof}[Proof of \cref{prop:componentsize}]
    Let $G$ be a $d$-regular graph, let $F$ be a minimum linear forest of $G$, and let $X \subseteq \End(F)$. We proceed by induction on $\ell \coloneqq \abs{\End(F) \setminus X}$. The result clearly holds for $\ell = 0$, so suppose that $\ell > 0$.
    
    Consider the following iterative process. Let $X_1 \coloneqq X$ and let $F_1 \coloneqq F$. Then, given $X_i$ and $F_i$, let $\Delta_i \coloneqq \Delta(C(F_i,X_i))$ be the maximum degree of the subgraph induced by $C(F_i,X_i)$, and choose any vertex $v_i \in C(F_i,X_i)$ with degree $\Delta_i$ in $C(F_i,X_i)$. Define $X_{i+1} \coloneqq X_i \cup \{v_i\}$ and let $F_{i+1} \in \cF(F_i,X_i)$ be a linear forest with $X_{i+1} \subseteq \End(F_{i+1})$. This process gives a chain of rotation-components
    \[
        C(F,X) = C(F_1,X_1) \supseteq C(F_2,X_2) \supseteq C(F_3,X_3) \supseteq \dots \supseteq C(F_{\ell+1},X_{\ell+1}) = \emptyset.
    \]
    Since the maximum degree of the subgraph induced by $C(F_i,X_i)$ is $\Delta_i$, this implies that
    \begin{equation}
        e(C(F,X)) \le \sum_{i = 1}^\ell \abs{C(F_i,X_i) \setminus C(F_{i+1},X_{i+1})} \cdot \Delta_i. \label{eq:edgessumdelta}
    \end{equation}
    If $\abs{C(F_{i+1},X_{i+1})} \le \abs{C(F,X)} - i \cdot (d+1)/4$, then we know by the induction hypothesis that
    \begin{align*}
        \abs{C(F,X)} & \ge \frac{d+1}{4} \cdot i + \abs{C(F_{i+1},X_{i+1})} \\
        & \ge \frac{d+1}{4} \cdot i + \frac{d+1}{4} \cdot \abs{\End(F_{i+1}) \setminus X_{i+1}} = \frac{d+1}{4} \cdot \abs{\End(F) \setminus X}.
    \end{align*}
    So, we may assume that $\abs{C(F_{i+1},X_{i+1})} > \abs{C(F,X)} - i \cdot (d+1)/4$. Since $\Delta_1 \ge \dots \ge \Delta_k$, this implies that
    \begin{equation}
        \sum_{i = 1}^\ell \abs{C(F_i,X_i) \setminus C(F_{i+1},X_{i+1})} \cdot \Delta_i < \sum_{i = 1}^\ell \frac{d+1}{4} \cdot \Delta_i = \frac{d+1}{4} \sum_{i=1}^\ell \abs{N_G(v_i) \cap C(F_i,X_i)}. \label{eq:sumdeltasizeneighbours}
    \end{equation}
    We claim that
    \begin{equation}
        \abs{\End(F_{\ell+1}) \setminus X} + \sum_{i=1}^\ell \abs{N_G(v_i) \cap C(F_i,X_i)} \le \abs{C_1(F_{\ell+1},X)} + 2 \cdot \abs{C_0(F_{\ell+1},X)}. \label{eq:sizeneighbourssizecomponent}
    \end{equation}
    To see this, observe that every vertex on the left side of the inequality is contained in $C(F,X) = C(F_{\ell+1},X)$, and consider the following cases:
    \begin{enumerate}[(a)]
        \item Suppose that $u \in C_2(F_{\ell+1},X)$. Clearly $u \notin \End(F_{\ell+1}) \setminus X$ since $u$ has two neighbours on its path in $F_{\ell+1}$. Moreover, we claim that $u \notin N_G(v_i) \cap C(F_i,X_i)$ for all $i$. Indeed, if $u \in N_G(v_i)$ for some $i$, then by \labelcref{obs:rotateonetotwo} there exists at least one neighbour $w \in N_{F_{\ell+1}}(u)$ such that we could rotate $v_i$ to $w$ using the pivot $u$. This would imply $w \in C(F_{\ell+1},X)$, contradicting $u \in C_2(F_{\ell+1},X)$.
        \item If $u \in C_1(F_{\ell+1},X)$, we claim that either $u \in \End(F_{\ell+1}) \setminus X$ and $u \notin N_G(v_i) \cap C(F_i,X_i)$ for all $i$, or $u \notin \End(F_{\ell+1}) \setminus X$ and $u \in N_G(v_i) \cap C(F_i,X_i)$ for up to one value of $i$. Indeed, let $w \in N_{F_{\ell+1}}(u)$ be the unique neighbour of $u$ in $B(F_{\ell+1},X)$. If $u \in \End(F_{\ell+1}) \setminus X$ and $u \in N_G(v_i)$ for some $i$, then by \labelcref{obs:rotateonetoone} we could rotate $v_i$ to $w$ using the pivot $u$, contradicting $w \notin C(F_{\ell+1},X)$. Moreover, if $u \in N_G(v_i) \cap N_G(v_j)$ for some $i \neq j$, then by \labelcref{obs:rotatetwotoone} we could rotate one of the two endpoints $v_i$ or $v_j$ to $w$ using the pivot $u$, again contradicting $w \notin C(F_{\ell+1},X)$. This proves the claim.
        \item If $u \in C_0(F_{\ell+1},X)$, we claim that either $u \in \End(F_{\ell+1}) \setminus X$ and $u \in N_G(v_i) \cap C(F_i,X_i)$ for up to one value of $i$, or $u \notin \End(F_{\ell+1}) \setminus X$ and $u \in N_G(v_i) \cap C(F_i,X_i)$ for up to two values of $i$. Indeed, if $u \in \End(F_{\ell+1}) \setminus X$ and $u \in N_G(v_i) \cap N_G(v_j)$ for some $i \neq j$, then by \labelcref{obs:reducetwotoone} we could reduce the number of paths of $F_{\ell+1}$ by adding one of the two edges $u v_i$ or $u v_j$, contradicting that $F_{\ell+1}$ is a minimum linear forest. Moreover, if $u \in N_G(v_i) \cap N_G(v_j)$ for some $i < j$, then $u \notin C(F_k,X_k)$ for any $k > j$. Otherwise, we could find a linear forest $F' \in \cF(F_k,X_k)$ with $\{u, v_i, v_j\} \subseteq \End(F')$, but then by \labelcref{obs:reducetwotoone} we could reduce the number of paths of $F'$ by adding one of the two edges $u v_i$ or $u v_j$, contradicting that $F'$ is a minimum linear forest. This proves the claim.
    \end{enumerate}
    In all cases, the contribution of a vertex to the left side of inequality \labelcref{eq:sizeneighbourssizecomponent} is at most as large as its contribution to the right side, which proves this inequality.

    We may assume that $\abs{\End(F_{\ell+1}) \setminus X} \cdot (d+1)/4 = \abs{\End(F) \setminus X} \cdot (d+1)/4 \ge \abs{C(F,X)}$, as we would be done otherwise. Since $C(F,X) = C(F_{\ell+1},X)$, we therefore have
    \begin{equation}
        \frac{d+1}{2} \cdot \abs{\End(F_{\ell+1}) \setminus X} \ge \abs{C(F,X)} = \abs{C(F_{\ell+1},X)} \ge \frac{1}{2} \cdot \abs{C_1(F_{\ell+1},X)} + \abs{C_0(F_{\ell+1},X)}. \label{eq:endpointssizecomponent}
    \end{equation}
    By combining the inequalities \labelcref{eq:edgessumdelta,eq:sumdeltasizeneighbours,eq:sizeneighbourssizecomponent,eq:endpointssizecomponent} and using again that $C(F,X) = C(F_{\ell+1},X)$, we get that
    \begin{align*}
        2 \cdot e(C(F_{\ell+1},X)) & < \frac{d+1}{2} \sum_{i=1}^\ell \abs{N_G(v_i) \cap C(F_i,X_i)} \\
        & \le \frac{d+1}{2} \cdot \abs{C_1(F_{\ell+1},X)} + (d+1) \cdot \abs{C_0(F_{\ell+1},X)} - \frac{d+1}{2} \cdot \abs{\End(F_{\ell+1}) \setminus X} \\
        & \le \frac{d}{2} \cdot \abs{C_1(F_{\ell+1},X)} + d \cdot \abs{C_0(F_{\ell+1},X)}.
    \end{align*}
    By \cref{lem:doublecount}, it follows that $e(C(F_{\ell+1},X),B_1(F_{\ell+1},X)) > \abs{B_1(F_{\ell+1},X)} \cdot d/2$. Since we know that $F_{\ell+1} \in \cF(F,X)$, by \cref{lem:b0empty} this is equivalent to $e(C(F,X),B_1(F,X)) > \abs{B_1(F,X)} \cdot d/2$. As in the proof of \cref{thm:smalllinearforestbipartite}, we can then find a vertex $v \in C(F,X)$ such that if $X' \coloneqq X \cup \{v\}$ and $F' \in \cF(F,X)$ is a linear forest with $X' \subseteq \End(F')$, then $\abs{C(F,X)} \ge \abs{C(F',X')} + (d+1)/2$.
    By the induction hypothesis, it then follows that
    \[
        \abs{C(F,X)} \ge \abs{C(F',X')} + \frac{d+1}{2} \ge \frac{d+1}{4} \cdot \abs{\End(F') \setminus X'} + \frac{d+1}{2} \ge \frac{d+1}{4} \cdot \abs{\End(F) \setminus X}. \qedhere
    \]
\end{proof}

\section{Linear arboricity up to logarithmic error}\label{sec:lineararboricity}

We now use the results from the previous section to show that the linear arboricity of any graph $G$ on $n$ vertices with maximum degree $\Delta$ is at most $\Delta/2 + \cO(\log n)$. We achieve this by repeatedly choosing linear forests with few paths and removing them from the graph. If the maximum degree of the remaining graph after $\Delta/2$ iterations is $\cO(\log n)$, we can decompose the remaining edges into $\cO(\log n)$ linear forests, say with Vizing's theorem, to obtain the desired upper bound on the linear arboricity.

Note that the degree of a vertex after the first $\Delta/2$ iterations can only be larger than $\cO(\log n)$ if the vertex appeared as an endpoint in more than $\cO(\log n)$ linear forests during those iterations. To avoid this, we will choose the linear forests in each iteration randomly in such a way that every vertex has a low probability of being an endpoint. Then, we can simply use standard concentration inequalities and a union bound to bound the probability that any vertex is an endpoint in more than $\cO(\log n)$ linear forests, which completes the proof.

We start by showing that for every $d$-regular graph $G$, there exists a distribution on the minimum linear forests of $G$ such that every vertex has a low probability of being an endpoint. Recall that $\End(F)$ denotes the multiset of endpoints of a linear forest $F$. In the proof, we will make extensive use of the notation introduced in \cref{sec:smalllinearforests}.

\begin{lemma}\label{lem:lowendpointprobability}
    Let $G$ be a $d$-regular graph. Then, we can choose a random minimum linear forest $F$ of $G$ in such a way that for all vertices $v \in V(G)$ it holds that $\pr(v \in \End(F)) \le 16/(d+1)$.
\end{lemma}

\begin{proof}
    We claim that for all minimum linear forests $F$ of $G$ and all $X \subseteq \End(F)$, we can choose a random linear forest $F' \in \cF(F,X)$ in such a way that for all vertices $v \in V(G)$ it holds that $\pr(v \in \End(F') \setminus X) \le 16/(d+1)$. For $X = \emptyset$, this implies the result. We prove the claim by induction on $\ell \coloneqq \abs{\End(F) \setminus X}$. The claim clearly holds for $\ell = 0$, so suppose that $\ell > 0$.

    For all $u \in C(F,X)$, let $F_u \in \cF(F,X)$ be a linear forest with $X \cup \{u\} \subseteq \End(F_u)$. By the induction hypothesis, we know that we can choose a random linear forest $F_u' \in \cF(F_u,X \cup \{u\})$ such that for all vertices $v \in V(G)$ it holds that $\pr(v \in \End(F_u') \setminus (X \cup \{u\})) \le 16/(d+1)$. Choose a random linear forest $F'' \in \cF(F,X)$ by first choosing $u \in C(F,X)$ uniformly at random and then choosing the random linear forest $F_u'$. By the linearity of expectation, we know that
    \[
        \sum_{v \in C(F,X)} \pr(v \in \End(F'') \setminus X) = \ev(\abs{\End(F'') \setminus X}) = \abs{\End(F) \setminus X}.
    \]
    By \cref{prop:componentsize}, we also have $\abs{C(F,X)} \ge \abs{\End(F) \setminus X} \cdot (d+1)/4$. This implies that the set $U \coloneqq \{v \in C(F,X) : \pr(v \in \End(F'') \setminus X) \le 8/(d+1)\}$ must contain at least half of the vertices of $C(F,X)$.

    Now, choose the random linear forest $F' \in \cF(F,X)$ by first choosing $u \in U$ uniformly at random and then choosing the random linear forest $F_u'$. For all vertices $v \in V(G) \setminus U$, we have
    \begin{align*}
        \pr(v \in \End(F') \setminus X) & = \frac{1}{\abs{U}} \sum_{u \in U} \pr(v \in \End(F_u') \setminus X) \\
        & = \frac{1}{\abs{U}} \sum_{u \in U} \pr(v \in \End(F_u') \setminus (X \cup \{u\})) \le \frac{16}{d+1}.
    \end{align*}
    For all vertices $v \in U$, we have
    \begin{align*}
        \pr(v \in \End(F') \setminus X) & = \frac{1}{\abs{U}} \sum_{u \in U} \pr(v \in \End(F_u') \setminus X) \\
        & \le \frac{1}{\abs{U}} \sum_{u \in C(F,X)} \pr(v \in \End(F_u') \setminus X) \\
        & = \frac{\abs{C(F,X)}}{\abs{U}} \cdot \pr(v \in \End(F'') \setminus X) \le \frac{16}{d+1}. \qedhere
    \end{align*}
\end{proof}

We can now prove \cref{thm:lineararboricitylogerror}. As described earlier in this section, our strategy is to repeatedly remove linear forests from the graph $G$ using \cref{lem:lowendpointprobability}. Since that result only applies to regular graphs, we need to regularise $G$ after every step, but the remainder of the proof works exactly as outlined before.

\begin{proof}[Proof of \cref{thm:lineararboricitylogerror}]
    Let $G$ be a graph on $n$ vertices with maximum degree $\Delta$. Let $d \coloneqq \ceil{\Delta/2}$. By adding vertices and edges to $G$, we may extend $G$ to a $2d$-regular graph $G_d$.

    Consider the following iterative process. Start with the $2d$-regular graph $G_d$. Then, given a $2i$-regular graph $G_i$, let $C_i \subseteq G_i$ be a $2$-regular spanning subgraph of $G_i$. This subgraph exists because $G_i$ is regular with even degree. Choose a random minimum linear forest $F_i$ of $G_i$ according to \cref{lem:lowendpointprobability} so that for all vertices $v \in V(G_i)$ it holds that $\pr(v \in \End(F_i)) \le 16/(2i+1)$. Then, let $E_i \subseteq E(C_i)$ be the subset of all edges of $C_i$ that are incident to some endpoint of $F_i$. Note that every vertex of $G_i$ is incident to at least two edges of $F_i \cup E_i$. This implies that the subgraph $G_i \setminus (F_i \cup E_i)$ has maximum degree at most $2(i-1)$ and can therefore be extended to a $2(i-1)$-regular graph $G_{i-1}$.

    We iterate this process until we reach the empty graph $G_0$. Note that the restriction of each linear forest $F_i$ to the initial graph $G$ is a linear forest. We also have that every edge of $G$ is contained either in one of the linear forests $F_i$ or in one of the sets $E_i$. So, the degree of a vertex $v \in V(G)$ in $G' \coloneqq G \setminus (\bigcup_{i=1}^d F_i)$ is at most the number of edges of $\bigcup_{i=1}^d E_i$ that are incident to $v$.

    Let $D_v^i$ be the indicator random variable of the event that $v$ is incident to an edge of $E_i$. Since $v$ is incident to at most two edges of $E_i$, the degree of $v$ in $G'$ is at most $D_v \coloneqq \sum_{i=1}^d 2 D_v^i$. Moreover, $v$ can only be incident to an edge of $E_i$ if $v$ or one of its two neighbours in $C_i$ is an endpoint of $F_i$. By construction, the probability that a vertex is an endpoint of $F_i$ is at most $16/(2i+1) \le 8/i$, regardless of the choices of $F_{i+1}, \dots, F_d$. It follows that $\pr(D_v^i = 1) \le 24/i$, even conditioned on the values of $D_v^{i+1}, \dots, D_v^d$. So, the random variables $D_v^1, \dots, D_v^d$ are stochastically dominated by a collection of independent Bernoulli random variables $X_1, \dots, X_d$ with $\pr(X_i = 1) = 24/i$. In particular, $X \coloneqq \sum_{i=1}^d 2 X_i$ satisfies $D_v \le X$. Since $\mu \coloneqq \ev(X) \le 48 (\log d + 1)$, a standard Chernoff bound with $\delta \coloneqq 48 (\log n + 1) / \mu \ge 1$ implies that
    \[
        \pr(D_v > (1 + \delta) \mu) \le \pr(X > (1 + \delta) \mu) \le e^{-\frac{\delta^2 \mu}{2 + \delta}} \le e^{-\frac{\delta \mu}{3}} = e^{-16 (\log n + 1)} < \frac{1}{n}.
    \]
    So, a union bound shows that with positive probability, $D_v \le (1 + \delta) \mu = \cO(\log n)$ for all $v \in V(G)$. In that case, the maximum degree of $G'$ is at most $\cO(\log n)$ which implies that $G'$ can be decomposed into at most $\cO(\log n)$ linear forests. Together with $F_1, \dots, F_d$, this provides a decomposition of $G$ into at most $d + \cO(\log n) = \Delta/2 + \cO(\log n)$ linear forests.
\end{proof}

The proof for fractional linear arboricity follows almost the same strategy, except that it ignores vertices with degree larger than $\cO(\log \Delta)$ in the remaining graph. This produces a family of $\Delta/2 + \cO(\log \Delta)$ linear forests such that each edge is contained in that family with high probability. If we then choose a uniformly random linear forest from that family, this proves the upper bound on the fractional linear arboricity.

\begin{proof}[Proof of \cref{thm:fractionallineararboricitylogerror}]
    Let $G$ be a graph on $n$ vertices with maximum degree $\Delta$. Define $d$, $F_i$, $G'$, $D_v$, and $X$ as in the proof of \cref{thm:lineararboricitylogerror}. Since $\mu \coloneqq \ev(X) \le 48 (\log d + 1)$, a standard Chernoff bound with $\delta \coloneqq 48 (\log d + 1) / \mu \ge 1$ implies that
    \[
        \pr(D_v > (1 + \delta) \mu) \le \pr(X > (1 + \delta) \mu) \le e^{-\frac{\delta^2 \mu}{2 + \delta}} \le e^{-\frac{\delta \mu}{3}} = e^{-16 (\log d + 1)} < \frac{1}{d^2}.
    \]
    So, the probability that $v$ has degree larger than $c \coloneqq (1 + \delta) \mu = \cO(\log d)$ in $G'$ is less than $1/d^2$. Let $E \subseteq E(G')$ be the subset of all edges of $G'$ that are incident to a vertex of degree larger than $c$ in $G'$. In particular, every edge of $G$ is contained in $E$ with probability at most $2/d^2$. Moreover, the graph $G'' \coloneqq G' \setminus E$ has maximum degree $c$ which implies that $G''$ can be decomposed into at most $c+1$ linear forests $F_1', \dots, F_{c+1}'$. Choose a linear forest $F \in \{F_1, \dots, F_d, F_1', \dots, F_{c+1}'\}$ uniformly at random. Then, for every edge $e \in E(G)$ of $G$, we have
    \[
        \pr(e \in F) = \pr(e \notin E) \cdot \pr(e \in F \mid e \notin E) \ge \left(1 - \frac{2}{d^2}\right) \frac{1}{d + c + 1} \ge \frac{1}{d + \cO(\log d)}.
    \]
    So, the fractional linear arboricity of $G$ is at most $d + \cO(\log d) = \Delta/2 + \cO(\log \Delta)$.
\end{proof}

\section{Open problems}\label{sec:openproblems}

By proving the conjecture of Feige and Fuchs, we removed a major obstacle towards resolving the Linear Arboricity Conjecture. However, the bound on the linear arboricity that we obtained in \cref{thm:lineararboricitylogerror} has an error term of $\cO(\log n)$, which depends on $n$. A natural next step is to replace this error term with a term polynomial in $\log \Delta$. To achieve this, one would like to replace the union bound that we use with a variant of the Lov\'asz Local Lemma or a similar technique. Unfortunately, the dependency structure of the corresponding bad events is extremely intricate, and so it is not obvious how to make such a strategy work. We believe that it would already be interesting to provide such an improvement for bipartite graphs.

\begin{problem}
    Do all bipartite graphs $G$ with maximum degree $\Delta$ satisfy $\la(G) \le \Delta/2 + (\log \Delta)^{\cO(1)}$?
\end{problem}

In bipartite graphs, we can sample a random linear forest more easily than for non-bipartite graphs. Indeed, as shown in the proof of \cref{thm:smalllinearforestbipartite}, we can always find a vertex $u \in B_1(F,X)$ with at least $(d+1)/2$ neighbours in $C(F,X)$. If we then fix an endpoint of the linear forest at a uniformly random neighbour of $u$ in $C(F,X)$, all neighbours of $u$ will afterwards be removed from $C(F,X)$. Repeating this process results in a random linear forest where every vertex is an endpoint with probability at most $2/(d+1)$.

When we iteratively remove linear forests from a graph, in later iterations we would like to fix endpoints only at vertices that have not been endpoints in too many previous iterations. This would be possible if we never fix too many endpoints in the neighbourhood of a vertex since we could then simply avoid these vertices when generating a new linear forest. Note that the above random linear forest has a very low probability of fixing many endpoints in the neighbourhood of a given vertex. However, we cannot apply the Lov\'asz Local Lemma because the vertex $u$ as well as the subset of its neighbours that are contained in $C(F,X)$ can depend on previously fixed endpoints, creating large dependencies. Even just the fact that we only have a subset of the neighbours of $u$ available when we fix an endpoint seems to make this problem difficult. More precisely, even if we knew in advance the sequence of vertices $u$ obtained while fixing endpoints, but not the subset of available neighbours, we do not know how to ensure that we fix at most $(\log \Delta)^{\cO(1)}$ endpoints in the neighbourhood of any vertex.

We believe that the following online matching game exhibits the same difficulty and could be of independent interest. Suppose that $G$ is a bipartite graph with bipartition $U \cup W$ such that every vertex in $U$ has degree $c \cdot d$ while every vertex in $W$ has degree at most $d$. Let $U = \{u_1, \dots, u_n\}$. Now, for $i = 1, \dots, n$, we receive a subset $U_i \subseteq N(u_i)$ of size $c \cdot d / 2$, and we immediately have to match $u_i$ to an unmatched vertex in $U_i$. If the sets $U_i$ were determined in advance and $c \ge 2$, Hall's theorem would imply that we could find a perfect matching.

\begin{problem}
    If $c$ is a sufficiently large constant, is there a way to construct a perfect matching in this online matching game where an adversary chooses the subsets $U_i$ based on the matching of $u_1, \dots, u_{i-1}$?
\end{problem}

An alternative approach to replacing the error term of $\cO(\log n)$ with $(\log \Delta)^{\cO(1)}$ could be to prove stronger properties about how $C(F,X)$ evolves when fixing endpoints. This might be enough to recover the limited dependencies necessary to apply the Lov\'asz Local Lemma.

In terms of the existence of linear forests with few paths, the conjecture of Magnant and Martin remains open up to a factor of two, but we believe that our techniques could be sufficient to eventually resolve this conjecture completely. Note that even for regular bipartite graphs where we proved this conjecture, our result still differs by a factor of roughly two from the best lower bound of $n/(2d)$ that comes from a disjoint union of complete bipartite graphs.

There are also several potential generalisations of the conjecture of Feige and Fuchs that remain open. For instance, this conjecture might generalise to cycles\footnote{Here, an individual edge is considered to be a cycle of length two.} and to directed regular graphs.

\begin{conjecture}
    Every directed $d$-regular $n$-vertex graph has a collection of $\cO(n/d)$ vertex-disjoint cycles that cover all vertices of the graph.
\end{conjecture}

The authors of the current paper \cite{christoph2025cyclefactors} very recently proved that there exists a collection of $\cO((n \log d)/d)$ cycles that satisfy this, which is the current best upper bound. In addition, they provided a randomised polynomial-time algorithm that computes such a collection. As a corollary, they also obtained an efficient algorithm for computing a tour of length $(1 + \cO((\log d)/d)) \cdot n$ in a connected $d$-regular $n$-vertex graph.

In contrast, all of the results in this paper are non-constructive, and so we do not obtain efficient algorithms. This raises the question whether there are efficient algorithms for finding spanning linear forests with $\cO(n/d)$ paths in $d$-regular $n$-vertex graphs. This would also imply the existence of efficient algorithms for computing a tour of length $(1 + \cO(1/d)) \cdot n$ in connected regular graphs. An efficient algorithm would likely improve our understanding of small linear forests significantly, which may lead to further progress on the Linear Arboricity Conjecture.

\begin{problem}\label{prob:linearforestefficientalgorithm}
    Is there an algorithm that can efficiently compute a spanning linear forest of a $d$-regular $n$-vertex graph with at most $\cO(n/d)$ paths?
\end{problem}

In particular, what is the computational complexity of finding a rotation-component $C(F,X)$? This could plausibly be NP-hard. In extremal combinatorics it is unusual for a proof to rely on hard-to-compute objects, but this does occur with algebraic methods such as the method of interlacing polynomials \cite{marcus2013interlacing,marcus2015interlacing} and topological methods such as applications of the Borsuk-Ulam Theorem \cite{matouvsek2003using}. Note that for Hamiltonian cubic graphs where minimum linear forests are Hamilton paths, it is known that the number of rotations required to rotate one endpoint to some fixed vertex of the graph may be exponential in $n$. This follows, for example, from a construction in \cite{brianski2022short}.

\bibliography{bib}

\newcommand{\etalchar}[1]{$^{#1}$}
\begin{thebibliography}{DMC{\etalchar{+}}24}
\providecommand{\url}[1]{\texttt{#1}}
\providecommand{\urlprefix}{\textsc{url:} }
\expandafter\ifx\csname urlstyle\endcsname\relax
  \providecommand{\doi}[1]{doi:\discretionary{}{}{}#1}\else
  \providecommand{\doi}{doi:\discretionary{}{}{}\begingroup \urlstyle{rm}\Url}\fi

\bibitem[AEH80]{akiyama1980covering}
\textsc{Jin Akiyama}, \textsc{Geoffrey Exoo}, and \textsc{Frank Harary} (1980).
\newblock \href{http://dml.cz/dmlcz/136252}{Covering and packing in graphs. {III}. {C}yclic and acyclic invariants}.
\newblock \emph{Mathematica Slovaca} \textbf{30}(4), 405--417.

\bibitem[AEH81]{akiyama1981covering}
\textsc{Jin Akiyama}, \textsc{Geoffrey Exoo}, and \textsc{Frank Harary} (1981).
\newblock \href{https://doi.org/10.1002/net.3230110108}{Covering and packing in graphs. {IV}. {L}inear arboricity}.
\newblock \emph{Networks} \textbf{11}(1), 69--72.

\bibitem[Alo88]{alon1988linear}
\textsc{Noga Alon} (1988).
\newblock \href{https://doi.org/10.1007/BF02783300}{The linear arboricity of graphs}.
\newblock \emph{Israel Journal of Mathematics} \textbf{62}(3), 311--325.

\bibitem[AS16]{alon2016probabilistic}
\textsc{Noga Alon} and \textsc{Joel~H. Spencer} (2016).
\newblock The probabilistic method.
\newblock Fourth edn. (John Wiley \& Sons).

\bibitem[BKLO16]{barber2016edge}
\textsc{Ben Barber}, \textsc{Daniela K{\"u}hn}, \textsc{Allan Lo}, and \textsc{Deryk Osthus} (2016).
\newblock \href{https://doi.org/10.1016/j.aim.2015.09.032}{Edge-decompositions of graphs with high minimum degree}.
\newblock \emph{Advances in Mathematics} \textbf{288}, 337--385.

\bibitem[BM24]{bucic2024towards}
\textsc{Matija Buci\'c} and \textsc{Richard Montgomery} (2024).
\newblock \href{https://doi.org/10.1016/j.aim.2023.109434}{Towards the {E}rd{\H{o}}s-{G}allai cycle decomposition conjecture}.
\newblock \emph{Advances in Mathematics} \textbf{437}, Paper No. 109434.

\bibitem[BS22]{brianski2022short}
\textsc{Marcin Bria{\'n}ski} and \textsc{Adam Szady} (2022).
\newblock \href{https://doi.org/10.1016/j.disc.2021.112624}{A short note on graphs with long {T}homason chains}.
\newblock \emph{Discrete Mathematics} \textbf{345}(1), Paper No. 112624.

\bibitem[CDG{\etalchar{+}}25]{christoph2025cyclefactors}
\textsc{Micha Christoph}, \textsc{Nemanja Dragani{\'{c}}}, \textsc{Ant{\'{o}}nio Gir{\~{a}}o}, \textsc{Eoin Hurley}, \textsc{Lukas Michel}, and \textsc{Alp M{\"{u}}yesser} (2025).
\newblock \href{https://arxiv.org/abs/2507.19417}{Cycle-factors of regular graphs via entropy}.
\newblock \emph{2025 {IEEE} 66th {A}nnual {S}ymposium on {F}oundations of {C}omputer {S}cience}.
\newblock To appear.

\bibitem[CKL{\etalchar{+}}16]{csaba2016proof}
\textsc{B{\'e}la Csaba}, \textsc{Daniela K{\"u}hn}, \textsc{Allan Lo}, \textsc{Deryk Osthus}, and \textsc{Andrew Treglown} (2016).
\newblock \href{https://doi.org/10.1090/memo/1154}{Proof of the 1-factorization and {H}amilton decomposition conjectures}.
\newblock \emph{Memoirs of the American Mathematical Society} \textbf{244}(1154), v+164.

\bibitem[DCS24]{draganic2024pancyclicity}
\textsc{Nemanja Dragani{\'c}}, \textsc{David~Munh{\'a} Correia}, and \textsc{Benny Sudakov} (2024).
\newblock \href{https://doi.org/10.4171/JEMS/1546}{Pancyclicity of hamiltonian graphs}.
\newblock \emph{Journal of the European Mathematical Society} .

\bibitem[DGCS25]{draganic2025optimal}
\textsc{Nemanja Dragani\'c}, \textsc{Stefan Glock}, \textsc{David~Munh{\'a} Correia}, and \textsc{Benny Sudakov} (2025).
\newblock \href{https://doi.org/10.1090/proc/16924}{Optimal {H}amilton covers and linear arboricity for random graphs}.
\newblock \emph{Proceedings of the American Mathematical Society} \textbf{153}(3), 921--935.

\bibitem[Dir52]{dirac1952some}
\textsc{Gabriel~Andrew Dirac} (1952).
\newblock \href{https://doi.org/10.1112/plms/s3-2.1.69}{Some theorems on abstract graphs}.
\newblock \emph{Proceedings of the London Mathematical Society (3)} \textbf{2}, 69--81.

\bibitem[DMC{\etalchar{+}}24]{draganic2024hamiltonicity}
\textsc{Nemanja Dragani{\'c}}, \textsc{Richard Montgomery}, \textsc{David~Munh{\'a} Correia}, \textsc{Alexey Pokrovskiy}, and \textsc{Benny Sudakov} (2024).
\newblock \href{http://arxiv.org/abs/2402.06603}{Hamiltonicity of expanders: optimal bounds and applications}.
\newblock arXiv:2402.06603.

\bibitem[DP21]{delcourt2021progress}
\textsc{Michelle Delcourt} and \textsc{Luke Postle} (2021).
\newblock \href{https://doi.org/10.1016/j.jctb.2020.09.008}{Progress towards {N}ash-{W}illiams' conjecture on triangle decompositions}.
\newblock \emph{Journal of Combinatorial Theory, Series B} \textbf{146}, 382--416.

\bibitem[EGP66]{erdos1966representation}
\textsc{Paul Erd{\H{o}}s}, \textsc{A.~W. Goodman}, and \textsc{Lajos P{\'o}sa} (1966).
\newblock \href{https://doi.org/10.4153/CJM-1966-014-3}{The representation of a graph by set intersections}.
\newblock \emph{Canadian Journal of Mathematics} \textbf{18}, 106--112.

\bibitem[Eno81]{enomoto1981linear}
\textsc{Hikoe Enomoto} (1981).
\newblock The linear arboricity of 5-regular graphs.
\newblock Technical report, Department of Information Science, University of Tokyo.

\bibitem[EP84]{enomoto1984linear}
\textsc{Hikoe Enomoto} and \textsc{Bernard P\'eroche} (1984).
\newblock \href{https://doi.org/10.1002/jgt.3190080211}{The linear arboricity of some regular graphs}.
\newblock \emph{Journal of Graph Theory} \textbf{8}(2), 309--324.

\bibitem[ERT80]{erdos1979choosability}
\textsc{Paul Erd{\H{o}}s}, \textsc{Arthur~L. Rubin}, and \textsc{Herbert Taylor} (1980).
\newblock Choosability in graphs.
\newblock \emph{Congressus Numerantium} \textbf{26}, 125--157.

\bibitem[FF22]{feige2022path}
\textsc{Uriel Feige} and \textsc{Ella Fuchs} (2022).
\newblock \href{https://doi.org/10.1002/jgt.22830}{On the path partition number of 6-regular graphs}.
\newblock \emph{Journal of Graph Theory} \textbf{101}(3), 345--378.

\bibitem[FFJ20]{ferber2020towards}
\textsc{Asaf Ferber}, \textsc{Jacob Fox}, and \textsc{Vishesh Jain} (2020).
\newblock \href{https://doi.org/10.1016/j.jctb.2019.08.009}{Towards the linear arboricity conjecture}.
\newblock \emph{Journal of Combinatorial Theory, Series B} \textbf{142}, 56--79.

\bibitem[FRS14]{feige2014short}
\textsc{Uriel Feige}, \textsc{R.~Ravi}, and \textsc{Mohit Singh} (2014).
\newblock \href{https://doi.org/10.1007/978-3-319-07557-0_23}{Short tours through large linear forests}.
\newblock \emph{Integer Programming and Combinatorial Optimization, IPCO 2014}, \emph{Lecture Notes in Computer Science}, vol. 8494, 273--284.

\bibitem[GCS24]{glock2024hamilton}
\textsc{Stefan Glock}, \textsc{David~Munh{\'a} Correia}, and \textsc{Benny Sudakov} (2024).
\newblock \href{https://doi.org/10.1016/j.aim.2024.109984}{Hamilton cycles in pseudorandom graphs}.
\newblock \emph{Advances in Mathematics} \textbf{458}, Paper No. 109984.

\bibitem[GKLO23]{glock2023existence}
\textsc{Stefan Glock}, \textsc{Daniela K{\"u}hn}, \textsc{Allan Lo}, and \textsc{Deryk Osthus} (2023).
\newblock \href{https://doi.org/10.1090/memo/1406}{The existence of designs via iterative absorption: hypergraph {$F$}-designs for arbitrary {$F$}}.
\newblock \emph{Memoirs of the American Mathematical Society} \textbf{284}(1406), v+131.

\bibitem[GKO16]{glock2016optimal}
\textsc{Stefan Glock}, \textsc{Daniela K\"uhn}, and \textsc{Deryk Osthus} (2016).
\newblock \href{https://doi.org/10.1016/j.jctb.2016.01.004}{Optimal path and cycle decompositions of dense quasirandom graphs}.
\newblock \emph{Journal of Combinatorial Theory, Series B} \textbf{118}, 88--108.

\bibitem[GL21]{gruslys2021cycle}
\textsc{Vytautas Gruslys} and \textsc{Shoham Letzter} (2021).
\newblock \href{https://doi.org/10.1017/s0963548320000553}{Cycle partitions of regular graphs}.
\newblock \emph{Combinatorics, Probability and Computing} \textbf{30}(4), 526--549.

\bibitem[GS05]{gerke2005sparse}
\textsc{Stefanie Gerke} and \textsc{Angelika Steger} (2005).
\newblock \href{https://doi.org/10.1017/CBO9780511734885.010}{The sparse regularity lemma and its applications}.
\newblock \emph{Surveys in combinatorics 2005}, \emph{London Mathematical Society Lecture Note Series}, vol. 327, 227--258.

\bibitem[GS24]{gao2024linear}
\textsc{Yuping Gao} and \textsc{Songling Shan} (2024).
\newblock \href{http://arxiv.org/abs/2405.18494}{Linear arboricity of graphs with large minimum degree}.
\newblock arXiv:2405.18494.

\bibitem[Gul86a]{guldan1986linear}
\textsc{Filip Guldan} (1986).
\newblock \href{http://dml.cz/dmlcz/132179}{The linear arboricity of {$10$}-regular graphs}.
\newblock \emph{Mathematica Slovaca} \textbf{36}(3), 225--228.

\bibitem[Gul86b]{guldan1986some}
\textsc{Filip Guldan} (1986).
\newblock \href{https://doi.org/10.1002/jgt.3190100408}{Some results on linear arboricity}.
\newblock \emph{Journal of Graph Theory} \textbf{10}(4), 505--509.

\bibitem[Han18]{han2018vertex}
\textsc{Jie Han} (2018).
\newblock \href{https://doi.org/10.37236/7109}{On vertex-disjoint paths in regular graphs}.
\newblock \emph{Electronic Journal of Combinatorics} \textbf{25}(2), Paper No. 2.12.

\bibitem[Har70]{harary1970covering}
\textsc{Frank Harary} (1970).
\newblock \href{https://doi.org/10.1111/j.1749-6632.1970.tb56470.x}{Covering and packing in graphs. {I}}.
\newblock \emph{Annals of the New York Academy of Sciences} \textbf{175}(1), 198--205.

\bibitem[Kah00]{kahn2000asymptotics}
\textsc{Jeff Kahn} (2000).
\newblock \href{https://doi.org/10.1002/1098-2418(200009)17:2<117::aid-rsa3>3.0.co;2-9}{Asymptotics of the list-chromatic index for multigraphs}.
\newblock \emph{Random Structures \& Algorithms} \textbf{17}(2), 117--156.

\bibitem[Kee14]{keevash2014existence}
\textsc{Peter Keevash} (2014).
\newblock \href{http://arxiv.org/abs/1401.3665}{The existence of designs}.
\newblock arXiv:1401.3665.

\bibitem[KKK{\etalchar{+}}23]{kang2021graph}
\textsc{Dong~Yeap Kang}, \textsc{Tom Kelly}, \textsc{Daniela K\"uhn}, \textsc{Abhishek Methuku}, and \textsc{Deryk Osthus} (2023).
\newblock \href{https://doi.org/10.4171/8ECM/11}{Graph and hypergraph colouring via nibble methods: {A} survey}.
\newblock \emph{European {C}ongress of {M}athematics}, vol.~21, 771--823.

\bibitem[KKOT19]{kim2019blow}
\textsc{Jaehoon Kim}, \textsc{Daniela K{\"u}hn}, \textsc{Deryk Osthus}, and \textsc{Mykhaylo Tyomkyn} (2019).
\newblock \href{https://doi.org/10.1090/tran/7411}{A blow-up lemma for approximate decompositions}.
\newblock \emph{Transactions of the American Mathematical Society} \textbf{371}(7), 4655--4742.

\bibitem[Koh97]{kohayakawa1997szemeredi}
\textsc{Y.~Kohayakawa} (1997).
\newblock Szemer{\'e}di's regularity lemma for sparse graphs.
\newblock \emph{Foundations of {C}omputational {M}athematics ({R}io de {J}aneiro, 1997)}, 216--230.

\bibitem[KS22]{keevash2022generalised}
\textsc{Peter Keevash} and \textsc{Katherine Staden} (2022).
\newblock \href{https://doi.org/10.1016/j.jctb.2021.09.007}{The generalised {O}berwolfach problem}.
\newblock \emph{Journal of Combinatorial Theory, Series B} \textbf{152}, 281--318.

\bibitem[KS25]{keevash2025ringel}
\textsc{Peter Keevash} and \textsc{Katherine Staden} (2025).
\newblock \href{https://doi.org/10.4171/jems/1604}{Ringel's tree packing conjecture in quasirandom graphs}.
\newblock \emph{Journal of the European Mathematical Society (JEMS)} \textbf{27}(5), 1769--1826.

\bibitem[KSSS24]{kwan2024high}
\textsc{Matthew Kwan}, \textsc{Ashwin Sah}, \textsc{Mehtaab Sawhney}, and \textsc{Michael Simkin} (2024).
\newblock \href{https://doi.org/10.4007/annals.2024.200.3.4}{High-girth {S}teiner triple systems}.
\newblock \emph{Annals of Mathematics} \textbf{200}(3), 1059--1156.

\bibitem[LMS25]{letzter2025nearly}
\textsc{Shoham Letzter}, \textsc{Abhishek Methuku}, and \textsc{Benny Sudakov} (2025).
\newblock \href{http://arxiv.org/abs/2503.07147}{Nearly {H}amilton cycles in sublinear expanders, and applications}.
\newblock arXiv:2503.07147.

\bibitem[LP23]{lang2023improved}
\textsc{Richard Lang} and \textsc{Luke Postle} (2023).
\newblock \href{https://doi.org/10.1007/s00493-023-00024-9}{An improved bound for the linear arboricity conjecture}.
\newblock \emph{Combinatorica} \textbf{43}(3), 547--569.

\bibitem[Mat03]{matouvsek2003using}
\textsc{Ji{\v{r}}{\'{i}} Matou{\v{s}}ek} (2003).
\newblock \href{https://doi.org/10.1007/978-3-540-76649-0}{Using the {B}orsuk-{U}lam theorem. {L}ectures on topological methods in combinatorics and geometry}.
\newblock Universitext (Springer, Berlin).

\bibitem[MM09]{magnant2009note}
\textsc{Colton Magnant} and \textsc{Daniel~M. Martin} (2009).
\newblock \href{https://ajc.maths.uq.edu.au/pdf/43/ajc_v43_p211.pdf}{A note on the path cover number of regular graphs}.
\newblock \emph{The Australasian Journal of Combinatorics} \textbf{43}, 211--217.

\bibitem[MMPS24]{montgomery2024approximate}
\textsc{Richard Montgomery}, \textsc{Alp M{\"u}yesser}, \textsc{Alexey Pokrovskiy}, and \textsc{Benny Sudakov} (2024).
\newblock \href{http://arxiv.org/abs/2406.02514}{Approximate path decompositions of regular graphs}.
\newblock \emph{Journal of the London Mathematical Society.} To appear.

\bibitem[Mon24]{montgomery2024transversals}
\textsc{Richard Montgomery} (2024).
\newblock \href{https://doi.org/10.1017/9781009490559.006}{Transversals in {L}atin squares}.
\newblock \emph{Surveys in combinatorics 2024}, \emph{London Mathematical Society Lecture Note Series}, vol. 493, 131--158.

\bibitem[MPS21]{montgomery2021proof}
\textsc{Richard Montgomery}, \textsc{Alexey Pokrovskiy}, and \textsc{Benny Sudakov} (2021).
\newblock \href{https://doi.org/10.1007/s00039-021-00576-2}{A proof of {R}ingel's conjecture}.
\newblock \emph{Geometric and Functional Analysis} \textbf{31}(3), 663--720.

\bibitem[MR90]{mcdiarmid1990linear}
\textsc{Colin McDiarmid} and \textsc{Bruce Reed} (1990).
\newblock \href{https://doi.org/10.1002/rsa.3240010405}{Linear arboricity of random regular graphs}.
\newblock \emph{Random Structures \& Algorithms} \textbf{1}(4), 443--445.

\bibitem[MSS15a]{marcus2013interlacing}
\textsc{Adam~W. Marcus}, \textsc{Daniel~A. Spielman}, and \textsc{Nikhil Srivastava} (2015).
\newblock \href{https://doi.org/10.4007/annals.2015.182.1.7}{Interlacing families {I}: {B}ipartite {R}amanujan graphs of all degrees}.
\newblock \emph{Annals of Mathematics} \textbf{182}(1), 307--325.

\bibitem[MSS15b]{marcus2015interlacing}
\textsc{Adam~W. Marcus}, \textsc{Daniel~A. Spielman}, and \textsc{Nikhil Srivastava} (2015).
\newblock \href{https://doi.org/10.4007/annals.2015.182.1.8}{Interlacing families {II}: {M}ixed characteristic polynomials and the {K}adison--{S}inger problem}.
\newblock \emph{Annals of Mathematics} \textbf{182}(1), 327--350.

\bibitem[NW61]{nash1961edge}
\textsc{C.~St.\ J.~A. Nash-Williams} (1961).
\newblock \href{https://doi.org/10.1112/jlms/s1-36.1.445}{Edge-disjoint spanning trees of finite graphs}.
\newblock \emph{The Journal of the London Mathematical Society} \textbf{36}(1), 445--450.

\bibitem[NW64]{nash1964decomposition}
\textsc{C.~St.\ J.~A. Nash-Williams} (1964).
\newblock \href{https://doi.org/10.1112/jlms/s1-39.1.12}{Decomposition of finite graphs into forests}.
\newblock \emph{The Journal of the London Mathematical Society} \textbf{39}(1), 12--12.

\bibitem[P{\'{e}}r82]{peroche1982partition}
\textsc{Bernard P{\'{e}}roche} (1982).
\newblock On partition of graphs into linear forests and dissections.
\newblock Preprint.

\bibitem[P{\'o}s76]{posa1976hamiltonian}
\textsc{Lajos P{\'o}sa} (1976).
\newblock \href{https://doi.org/10.1016/0012-365X(76)90068-6}{Hamiltonian circuits in random graphs}.
\newblock \emph{Discrete Mathematics} \textbf{14}(4), 359--364.

\bibitem[Sco11]{scott2011szemeredi}
\textsc{Alexander Scott} (2011).
\newblock \href{https://doi.org/10.1017/S0963548310000490}{Szemer{\'e}di's regularity lemma for matrices and sparse graphs}.
\newblock \emph{Combinatorics, Probability and Computing} \textbf{20}(3), 455--466.

\bibitem[Sha49]{shannon1949theorem}
\textsc{Claude~E. Shannon} (1949).
\newblock \href{https://doi.org/10.1002/sapm1949281148}{A theorem on coloring the lines of a network}.
\newblock \emph{Journal of Mathematics and Physics} \textbf{28}(1-4), 148--151.

\bibitem[Tom82]{tomasta1982note}
\textsc{Pavel Tomasta} (1982).
\newblock \href{http://dml.cz/dmlcz/136298}{Note on linear arboricity}.
\newblock \emph{Mathematica Slovaca} \textbf{32}(3), 239--242.

\bibitem[Vis12]{vishnoi2012permanent}
\textsc{Nisheeth~K. Vishnoi} (2012).
\newblock \href{https://doi.org/10.1109/FOCS.2012.81}{A permanent approach to the traveling salesman problem}.
\newblock \emph{2012 {IEEE} 53rd {A}nnual {S}ymposium on {F}oundations of {C}omputer {S}cience}, 76--80.

\bibitem[Viz64]{vizing1964estimate}
\textsc{Vadim~G. Vizing} (1964).
\newblock On an estimate of the chromatic class of a {$p$}-graph.
\newblock \emph{Diskret. Analiz} \textbf{3}, 25--30.

\bibitem[Viz76]{vizing1976coloring}
\textsc{Vadim~G. Vizing} (1976).
\newblock Coloring the vertices of a graph in prescribed colors.
\newblock \emph{Diskret. Analiz} \textbf{29}, 3--10.

\bibitem[Wu99]{wu1999linear}
\textsc{Jian-Liang Wu} (1999).
\newblock \href{https://doi.org/10.1002/(sici)1097-0118(199906)31:2<129::aid-jgt5>3.0.co;2-a}{On the linear arboricity of planar graphs}.
\newblock \emph{Journal of Graph Theory} \textbf{31}(2), 129--134.

\bibitem[WW08]{wu2008linear}
\textsc{Jian-Liang Wu} and \textsc{Yu-Wen Wu} (2008).
\newblock \href{https://doi.org/10.1002/jgt.20305}{The linear arboricity of planar graphs of maximum degree seven is four}.
\newblock \emph{Journal of Graph Theory} \textbf{58}(3), 210--220.

\end{thebibliography}

\end{document}